\documentclass[10pt]{amsart}
 \usepackage{amssymb,stmaryrd,tikz,fullpage}
 \usepackage[all]{xy}
 \usepackage[bookmarks=false,pdfstartview={FitH}]{hyperref}

 \newdir{ >}{{}*!/-9pt/\dir{>}}
 \def\co{\colon\thinspace}
 \def\oc{:\!}
 
 \newtheorem{theorem}{Theorem}[section]
 \newtheorem{corollary}[theorem]{Corollary}
 \newtheorem{lemma}[theorem]{Lemma}
 \newtheorem{proposition}[theorem]{Proposition}
 \newtheorem{conjecture}[theorem]{Conjecture}
 \newtheorem{theorem*}{Theorem}

 \theoremstyle{definition}
 \newtheorem{definition}[theorem]{Definition}

 \numberwithin{equation}{section}

 \DeclareMathOperator{\IM}{Im}
 \DeclareMathOperator{\Exp}{exp}
 \DeclareMathOperator{\Hom}{Hom}
 \DeclareMathOperator{\Map}{Map}
 \DeclareMathOperator{\Dim}{dim}
 \DeclareMathOperator{\Ker}{Ker}
 \DeclareMathOperator{\End}{End}
 \DeclareMathOperator{\inj}{Inj}
 \DeclareMathOperator{\FMap}{FMap}
 
 \newcommand{\Real}{\mathbb{R}}
 \newcommand{\Complex}{\mathbb{C}}
 \newcommand{\redKstar}{\tilde{K}^{*}_{G}}
 \newcommand{\redKzero}{\tilde{K}^{0}_{G}}
 \newcommand{\Kstar}{K^{*}_{G}}
 \newcommand{\Kone}{K^{1}_{G}}
 \newcommand{\Kzero}{K^{0}_{G}}
 \newcommand{\Ell}{\mathcal{L}(V_{0},V_{1})}
 \newcommand{\aich}{\Hom(V_{0},V_{1})}
 \newcommand{\svo}{s(V_{0})} 
  
\begin{document}

\title{An Equivariant Generalization of the Miller Splitting Theorem}
\author{Harry Ullman}
\thanks{Supported by an EPSRC/UoS Doctoral Prize Fellowship. Part of this work was completed as a PhD student supported by the EPSRC}
\address{School of Mathematics and Statistics\\
University of Sheffield\\
Hicks Building\\
Hounsfield Road\\
Sheffield\\ 
S3 7RH\\
UK}
\email{h.ullman@sheffield.ac.uk}
\date{March 2011}
\keywords{isometry, Miller splitting, cofibre sequence, functional calculus, Gysin map, residue}
\subjclass{55P42,55P91,55P92}

\begin{abstract} Let $G$ be a compact Lie group. We build a tower of $G$--spectra over the suspension spectrum of the space of linear isometries from one $G$--representation to another. The stable cofibres of the maps running down the tower are certain interesting Thom spaces. We conjecture that this tower provides an equivariant extension of Miller's stable splitting of Stiefel manifolds. We provide a cohomological obstruction to the tower producing a splitting in most cases; however, this obstruction does not rule out a split tower in the case where the Miller splitting is possible. We claim that in this case we have a split tower which would then produce an equivariant version of the Miller splitting, we prove this claim in certain special cases though the general case remains a conjecture. To achieve these results we construct a variation of the functional calculus with useful homotopy-theoretic properties and explore the geometric links between certain equivariant Gysin maps and residue theory.
\end{abstract}

\maketitle 

\tableofcontents

\section{Introduction}\label{Introduction}

Let $G$ be a compact Lie group and let $V_{0}$ and $V_{1}$ be finite dimensional complex $G$--representations with $G$--invariant inner product such that $d_{0}:=\Dim(V_{0})\leqslant \Dim(V_{1})$. Let $\Ell$ be the space of all linear isometries from $V_{0}$ to $V_{1}$ equipped with the usual conjugation $G$--action. The aim of this paper is to study the equivariant stable homotopy theory of $\Ell$.

Let $\Ell_{+}$ be $\Ell$ equipped with a disjoint $G$--fixed basepoint. We construct a stable diagram containing $\Ell_{+}$ with interesting topological properties. We write $T$ for the tautological bundle over the equivariant Grassmannian $G_{k}(V_{0})$, use $\Hom(T,V_{1}-V_{0})$ as a shorthand for the virtual bundle $\Hom(T,V_{1})-\Hom(T,V_{0})$ and we let $s(T)$ be the bundle
\[
\{(V,\alpha):V\in G_{k}(V_{0}),\alpha\in s(V)\}.
\]

\begin{theorem*}\label{IntroTower} There is a natural tower of $G$--spectra
\[
\Ell_{+}\to X_{d_{0}-1}\to\ldots\to X_{1}\to S^{0}
\]
such that the stable homotopy fibres of the maps $X_{k}\to X_{k-1}$ are the Thom spaces
\[
G_{k}(V_{0})^{\Hom(T,V_{1}-V_{0})\oplus s(T)}.
\]
\end{theorem*}

The above result is phrased differently when proved in Section~\ref{An equivariant stable tower over isometries}; our statements there concern homotopy cofibres rather than homotopy fibres but we state the theorem using fibres here to avoid superfluous suspensions. We cover $G$--spectra in detail in Section~\ref{Conventions}, however, we note here that we use $G$--spectra indexed on a chosen complete $G$--universe, rather than naive $G$--spectra indexed over $\mathbb{Z}$.

Studying the cofibres of this tower leads to interesting homotopical insight about $\Ell$. In particular this result can be seen as generalization of Miller's stable splitting of Stiefel manifolds \cite{Miller}, we also refer the reader to \cite[Section $1$]{Crabb}, \cite[Section $1$]{Kitchloo} and \cite[Appendix $A$]{HarryPhD}. Consider the above setup without equivariance, then $V_{0}\cong \Complex^{d_{0}}$, $V_{1}\cong \Complex^{d_{0}+t}$ for some $t$ and we can think of $\mathcal{L}(\Complex^{d_{0}},\Complex^{d_{0}+t})$ as a Stiefel manifold. Miller showed that there is a stable splitting
\[
\mathcal{L}(\Complex^{d_{0}},\Complex^{d_{0}+t})_{+}\simeq \bigvee_{k=0}^{d_{0}}G_{k}(\Complex^{d_{0}})^{\Hom(T,\Complex^{t})\oplus s(T)}.
\]
We investigate whether our tower can produce a similar stable splitting. Returning to our equivariant setup, consider the case where $V_{0}$ is a subrepresentation of $V_{1}$. We conjecture that our tower splits to retrieve an equivariant form of the Miller splitting. We cannot show this, however we can show that the bottom and top of the tower split and thus we prove the following theorem. 

\begin{theorem*}\label{IntroSplitting} Let $V_{0}\leqslant V_{1}$ and let $d_{0}=2$, then we have a split tower and recover an equivariant Miller splitting
\[
\Ell_{+}\simeq \bigvee_{k=0}^{2}G_{k}(V_{0})^{\Hom(T,V_{1}-V_{0})\oplus s(T)}.
\]
\end{theorem*}

Return to the general case, where $V_{0}$ may not necessarily be a subrepresentation of $V_{1}$. We investigate whether the tower splits in the more general setting by studying interesting geometric properties satisfied by one of the maps in the tower. This investigation includes a treatment of the links between certain equivariant Gysin maps and residue theory; in particular we cover an interesting general result equivariantly extending previous study of Quillen \cite{QuillenCobordismFGL}.

\begin{theorem*}\label{IntroObstruction} Let $G$ be connected. There is a cohomological obstruction to the tower splitting if $V_{0}$ is not a subrepresentation of $V_{1}$. If $G$ is not connected then there is a cohomological obstruction to the tower splitting if the $K$--theory polynomial associated to $V_{0}$ does not divide the $K$--theory polynomial associated to $V_{1}$.
\end{theorem*}

To achieve these results we first build a variation of the functional calculus with useful homotopy-theoretic properties. The functional calculus is a tool from functional analysis that is used to construct elements of a $C^{*}$--algebra using continuous functions. Let $V$ be a Hermitian space, then set $s(V)$ to be the space of self-adjoint endomorphisms of $V$. We build a space $D(d)$ and subspaces $F_{i}(D(d))$ which model eigenvalues of elements of $s(V)$. We then build a continuous generalization of the functional calculus which takes a self-adjoint endomorphism $\alpha$ and a continuous self-map $f$ of $D(d)$ such that $f(F_{i}(D(d)))\subseteq F_{i}(D(d))$ and outputs a new self-adjoint endomorphism denoted $\mathfrak{A}_{f}(\alpha)$.

Further, we extend this construction from $s(V)$ to $\Hom(V,W)$ for $W$ another Hermitian space. Let $\gamma\in\Hom(V,W)$, then we can use this functional calculus to build a new homomorphism $\mathfrak{B}_{f}(\gamma)\co V\to W$. Let $S^{s(V)}$ denote the one-point compactification of $s(V)$ and let $S^{\Hom(V,W)}$ denote the one-point compactification of $\Hom(V,W)$. Then our functional calculus gives us maps
\[
\mathfrak{A}_{f}\co S^{s(V)}\to S^{s(V)}
\]
\[
\mathfrak{B}_{f}\co S^{\Hom(V,W)}\to S^{\Hom(V,W)}.
\]
The power of this construction comes from its homotopy properties. Many useful maps can be rephrased in the form $\mathfrak{A}_{f}$ or $\mathfrak{B}_{f}$. Let $f\co D(d)\to D(d)$ and $g\co D(d)\to D(d)$ be homotopic via a homotopy that preserves each $F_{i}(D(d))$. Then $\mathfrak{A}_{f}\simeq \mathfrak{A}_{g}$ and $\mathfrak{B}_{f}\simeq \mathfrak{B}_{g}$. Thus determining the homotopy type of $D(d)$ determines the homotopy type of maps built using this functional calculus. The intersection of all subspaces $F_{i}(D(d))$ is naturally a copy of $S^{1}$ sitting inside $D(d)$. We use the below result to prove Theorem~\ref{IntroTower}, amongst other statements.

\begin{theorem*}\label{IntroHomotopy} Let $f$ and $g$ be two self-maps of $D(d)$ such that $f(F_{i}(D(d)))\subseteq F_{i}(D(d))$ and $g(F_{i}(D(d)))\subseteq F_{i}(D(d))$ for all $i$. Then there are induced maps $f'$, $g'\co S^{1}\to S^{1}$ and $f$ and $g$ are homotopic via a homotopy that preserves each $F_{i}(D(d))$ if and only if $f'$ and $g'$ have the same degree.
\end{theorem*}

This paper is laid out as follows. Section~\ref{Conventions} covers various notational statements, conventions and technical statements we will use throughout the document. Section~\ref{Extended functional calculus} details an overview of our functional calculus variation, including a concrete example and concluding with a proof of Theorem~\ref{IntroHomotopy}. The main result, Theorem~\ref{IntroTower}, is stated in more detail and proved in Section~\ref{An equivariant stable tower over isometries}. This section also includes explicit statements regarding the maps in the tower. Section~\ref{Gysin maps and residues} begins with a general study of Gysin maps associated to equivariant embeddings of projective space. We provide geometric links between these maps and residue maps before using the general theory and the geometric properties of the bottom of the tower to prove Theorem~\ref{IntroObstruction} and provide a cohomological obstruction to a stable splitting in the general case. Section~\ref{The subrepresentation case---conjecture} covers the conjecture in the special case where $V_{0}$ is a subrepresentation of $V_{1}$, the only case where a splitting is possible. We then retrieve the dimension $2$ special case Theorem~\ref{IntroSplitting} by considering the compatibility of our work with Miller's work \cite{Miller}.

Many of the results in this paper were first detailed in the author's PhD thesis \cite{HarryPhD}---proofs left to the reader in this document are generally recorded in \cite{HarryPhD}. The author would like to thank his supervisor Neil Strickland for much support, advice and insight. 

\section{Conventions}\label{Conventions}

Our spaces are compactly generated weak Hausdorff $G$--spaces, when we have basepoints they are $G$--fixed. We pass from unbased spaces to based spaces via the Alexandroff one point-compactification; we denote the one-point compactification of $X$ by $X_{\infty}$ and take the basepoint to be the added point. This is equivalent when $X$ is compact to adjoining a disjoint basepoint, hence $X_{+}=X_{\infty}$ in this case and we mostly dispense with $X_{+}$ notation from this point onwards. We recall a map $f$ to be proper if and only if the inverse image of any compact set is compact. A proper map $f\co X\to Y$ then has a continuous extension $f_{\infty}\co X_{\infty}\to Y_{\infty}$. One other convention we use is that if $X'$ is an unbased space then $X$ tends to be used to denote the one-point compactification.

We assume $G$ acts on the left, let $\Map(X,Y)$ denote the space of continuous maps from $X$ to $Y$ equipped with the compact-open topology. We equip this and other mapping spaces with the conjugation group action $(g.f)(x)=gf(g^{-1}x)$.  For more exotic spaces we mention the action where appropriate, but note here that most are derivatives of a conjugation action. We skirt over most detailed statements regarding $G$--actions; these points are easy enough to check, repetitive and unenlightening. More detail can be found in \cite{HarryPhD}. 

We choose a complete $G$--universe and work in the homotopy category of $G$--spectra indexed on this universe. Our work then holds independently of the choice of model of the homotopy category. For example the results on spectra hold equally well for the spectra of \cite{LewisMaySteinberger}, equivariant $S$--modules or orthogonal spectra as in \cite{OrthogonalSpectraSModules}, or similar. This follows from the method---all one needs to construct the presented results is that the category of spectra we work in has cofibre sequences and a suspension spectrum functor $\Sigma^{\infty}$ that preserves cofibre sequences. Further, the main result can actually be viewed as a result in the equivariant stable category---this category is triangulated as shown in \cite[Section $9.4$]{HoveyPalmieriStrickland} with distinguished triangles built out of cofibre sequences.

We also have certain notational conventions that we use. Let $X$ and $Y$ be spaces and let $f\co X\to Y$. The cone on $X$ is $C(X):=[0,1]\wedge X$ with the convention that $[0,1]$ is based at $0$. The cofibre of $f$ is then $C_{f}:=C(X)\vee Y/((1,x)\sim f(x))$. We assume the twist in a cofibre sequence occurs as
\[
X\overset{f}{\to} Y\overset{\operatorname{inc.}}\to C_{f}\overset{\operatorname{coll}}{\to} \Sigma X\overset{-\Sigma f}{\to} \Sigma Y\to\ldots
\]
where $-\Sigma f$ is the map $(t,x)\mapsto (1-t,f(x))$, assuming for now that the suspension coordinate runs over $(0,1)$. Let $V$ and $W$ be vector spaces. We often work with the subspace $\inj(V,W)$ of $\Hom(V,W)$ consisting of injective homomorphisms. The subspace of $\Hom(V,W)$ of non-injective homomorphisms is denoted by $\inj(V,W)^{c}$. For $V$ and $W$ representations we use the notation $V\leqslant W$ to mean both vector subspace and, where appropriate, subrepresentation---which we mean at any given point will be clear from the context. We use the notation
\[
\xymatrix{X\ar[r]|\bigcirc&Y}
\]
to denote maps $X\to \Sigma Y$. Throughout the document we use various forms of exponential maps. We use the notation $\Exp(x)$ in most cases, however, if $x$ is just a number we tend to switch to $e^{x}$. In both cases the inverse is normally denoted $\log$. We also note the distinction between $\Real^{+}$ and $\Real^{++}$, the former is the space of nonnegative numbers while the latter is the space of strictly positive numbers.

We note here some conventions on $\Real$ and homeomorphic spaces. We implicitly assume throughout that whenever $\Real\cong \Real^{++}$ it is via $x\mapsto e^{x}$ and whenever $(0,1)\cong \Real$ it is via $x\mapsto \log(x/(1-x))$. Use of these homeomorphisms is generally not explicitly stated but each incident of implicit use should be clear.

Throughout the document we state many homeomorphisms (for example \ref{thetaEalpha}, \ref{Xkothertopology}, \ref{themaptau}, \ref{thetaEalphaextended}, \ref{sigmainjchomeoextended} and elsewhere) which seem to include a superfluous minus sign. This is a technical necessity that allows the work to blend well with the Miller splitting; compare \ref{thetaEalpha} to \cite[Lemma $1.1$]{Crabb} or \cite[Lemma $1.3$]{Kitchloo}. 

Finally, we remark on material omitted from this paper. Many proofs, as already noted, have been left to the reader. Most of the omitted detail is of three different forms. Firstly, as discussed above, much of the detail of equivariance is omitted. Secondly many of the omissions deal with simple fact checking---checking that compositions are identities, checking that maps land in the right codomains and checking some simple continuity arguments. Finally, we omit many properness arguments because they all have the same flavour. We tend to deal with maps between normed spaces or bundles over compact bases with normed fibres. In these cases the compact subsets are known to be the closed (fibrewise) bounded subspaces. Further the spaces we deal with are mostly Hausdorff, hence checking the closed property is a triviality as compact subspaces of Hausdorff spaces are closed as standard. Thus the arguments boil down to checking bounds---we assume that $\|f(x)\|$ is bounded and wish to find a bound on $\|x\|$. This is generally a simple exercise in inequalities, made even easier by noting that if a composition $g\circ f$ is proper then $f$ is proper. Thus we omit much of the work of this type. The omitted work can generally be found in \cite{HarryPhD}.

\subsection{Technical results}\label{Technical results}

In this section we gather together a few technical lemmas that we will use in the rest of the document. We mention sketches of many of the proofs but omit some of the detail, which if needed can be found in \cite{HarryPhD}. We advise the reader to skip this section and refer back to the results when needed.

\begin{lemma}\label{quotientandsubspacearethesame} Let $X$ and $Z$ be locally compact Hausdorff spaces and let $f\co
X\to Z$
be a continuous proper map. Setting $Y:=f(X)$, we have an inclusion $j\co Y\rightarrowtail Z$ and surjection $p\co X\twoheadrightarrow Y$ and setting $Y_{\infty}:=f_{\infty}(X_{\infty})$ we have extensions $j_{\infty}\co Y_{\infty}\rightarrowtail Z_{\infty}$ and
$p\co X_{\infty}\twoheadrightarrow Y_{\infty}$ such that the diagram of sets
\[
\xymatrix{X\ar@{->>}[r]^{p}\ar@{ >->}[d]_{i_{X}}&Y\ar@{ >->}[r]^{j}\ar@{ >->}[d]^{i_{Y}}&Z\ar@{ >->}[d]^{i_{Z}}\\
X_{\infty}\ar@{->>}[r]_{p_{\infty}}&Y_{\infty}\ar@{ >->}[r]_{j_{\infty}}&Z_{\infty}}
\]
commutes. Then there are unique topologies on $Y$ and $Y_{\infty}$ such that:
\begin{enumerate}
\item $Y$ is a locally compact Hausdorff space with one-point compactification $Y_{\infty}$.
\item $p$ is a proper quotient.
\item $j$ is a proper closed inclusion.
\item $p_{\infty}$ is a quotient.
\item $j_{\infty}$ is a closed inclusion.
\item $i_{Y}$ is an open inclusion.
\end{enumerate}
\end{lemma} 

This result can be proved by a standard point-set topology argument. We use it to demonstrate that certain spaces we construct have both a subspace topology and an equivalent quotient topology. This will then prove useful in simplifying some continuity arguments.

We assume throughout many standard facts about cofibre sequences---that they can be built from neighbourhood deformation retract pairs, that isomorphisms of cofibre sequences are isomorphisms in the homotopy category and that smashing a cofibre sequence with a space produces another cofibre sequence. We also assume the following result regarding the interactions between cofibre sequences and bundles.

\begin{lemma}\label{cofibresequenceandbundles} Let $A$ be a space and let $\{X_{a}\}_{a\in A}$, $\{Y_{a}\}_{a\in A}$ and $\{Z_{a}\}_{a\in A}$ be families of based spaces equipped with the following structure:
\begin{itemize}
\item Total spaces $X:=\bigcup_{a\in A} X_{a}$, $Y:=\bigcup_{a\in A} Y_{a}$ and $Z:=\bigcup_{a\in A} Z_{a}$.
\item Projections $X\overset{\pi_{1}}{\to} A$, $Y\overset{\pi_{2}}{\to} A$ and $Z\overset{\pi_{3}}{\to} A$ given by $x\in X_{a}\mapsto a$, $y\in Y_{a}\mapsto a$ and $z\in Z_{a}\mapsto a$.
\item Sections $A\overset{\sigma_{1}}{\to} X$, $A\overset{\sigma_{2}}{\to} Y$ and $A\overset{\sigma_{3}}{\to} Z$ sending $a$ to the basepoint in $X_{a}$, $Y_{a}$ or $Z_{a}$.
\end{itemize}
Let $\Sigma_{A}X:= \bigcup_{a\in A}\Sigma X_{a}$ and assume that there is a sequence of continuous maps $X\overset{f}{\to}Y\overset{g}{\to} Z\overset{h}{\to}\Sigma_{A} X$ arising from fibrewise cofibre sequences $X_{a}\overset{f_{a}}{\to} Y_{a}\overset{g_{a}}{\to} Z_{a}\overset{h_{a}}{\to} \Sigma X_{a}$. Then we have a cofibre sequence
\[
X/\sigma_{1}(A)\to Y/\sigma_{2}(A)\to Z/\sigma_{3}(A)\to \Sigma X/\sigma_{1}(A).
\]
\end{lemma}

This result roughly states that if we have a sequence of bundles that is a fibrewise cofibre sequence then it is a cofibre sequence. It can be proved from first principles. Finally, we state a result regarding quotients of cofibre sequences.

\begin{lemma}\label{quotientofacofibresequence} Let $Z$ include into both $X$ and $Y$ and let $f\co X\to Y$ be such that there is an induced map $\bar{f}\co X/Z\to Y/Z$. Further assume either
\begin{itemize}
\item the inclusions $Z\rightarrowtail X$ and $Z\rightarrowtail Y$ are cofibrations,
\item OR $X$, $Y$ and $Z$ are simply connected $CW$--complexes and $X$ and $Y$ are connected.
\end{itemize}
Then $C_{f}$ is naturally homotopy equivalent to $C_{\bar{f}}$.
\end{lemma}

Assume the first condition holds, then analysis of the diagram
\[
\xymatrix{Z\ar@{ >->}[d]\ar[r]^{1}&Z\ar@{ >->}[d]\ar[r]&C(Z)\ar[d]\\
X\ar[d]\ar[r]^{f}&Y\ar[d]\ar[r]&C_{f}\ar[d]\\
\frac{X}{Z}\ar[r]_{\bar{f}}&\frac{Y}{Z}\ar[r]&C_{\bar{f}}}
\]
leads to the result. A stable version of the result has an alternate proof using the octahedral axiom. The second version of the lemma relies on the theory of cubical diagrams, as outlined in \cite[Section $1$]{Goodwillie2}. The cited paper states many results for total fibres, dual results can be proved for total cofibres. The proof is begun by considering the diagram below.
\[
\xymatrix{Z\ar[rr]\ar[dd]\ar[dr]&&\text{pt}\ar[dr]\ar[dd]&\\
&X\ar[rr]\ar[dd]&&X/Z\ar[dd]\\
Z\ar[rr]\ar[dr]&&\text{pt}\ar[dr]&\\
&Y\ar[rr]&&Y/Z}
\]
This diagram has zero total cofibre as the top and bottom faces are homotopy pushouts, furthermore the rear face has zero cofibre and thus the cofibre of $C_{f}\to C_{\bar{f}}$ is zero. The connectedness assumptions are then needed to make the claimed conclusion. All of these proofs can be found in more detail in \cite[Section $2.3$]{HarryPhD}. We use this result to take a quotient at a certain point in Section~\ref{The cofibre sequences}, simplifying the work required to prove Theorem~\ref{IntroTower}.

\section{Extended functional calculus}\label{Extended functional calculus}

In this section we extend the theory of functional calculus, a tool originally developed in functional analysis. Our extension has interesting homotopy-theoretic properties which we will use in Section~\ref{An equivariant stable tower over isometries} to prove Theorem~\ref{IntroTower}. Let $V$ and $W$ be Hermitian spaces, i.e.\ complex vector spaces equipped with Hermitian inner products, such that $\dim(V)\leqslant \dim(W)$. We refer the reader to \cite[Appendix $A$]{StricklandSubbundles} for an overview of the original theory of functional calculus and we take as given knowledge of all results and statements made in \cite{StricklandSubbundles}. We also follow the conventions taken in the referenced paper, though we make three notational changes---we use $^{\dag}$ instead of $^{*}$ for adjoint, we use $s(V)$ rather than $w(V)$ for the space of self-adjoint endomorphisms of $V$ and if $\alpha\in s(V)$ we denote the eigenvalues of alpha (which are real numbers as standard) by $e_{0}(\alpha)\leqslant e_{1}(\alpha)\leqslant\ldots$ ordered by the standard $\leqslant$ ordering on $\Real$. Our norms on spaces of linear maps are assumed to be operator norms.

Let $s_{+}(V)$ be the space of self-adjoint endomorphisms of $V$ with non-negative eigenvalues and let $s_{++}(V)$ be the space of all self-adjoint endomorphisms of $V$ with strictly positive eigenvalues. All of the following constructions can be built from the functional calculus as standard, we leave details of the proofs up to the reader.

\begin{lemma}\label{ExpAndLog} We have $s(V)\cong s_{++}(V)$ via
\begin{align*}
\Exp\co s(V)&\to s_{++}(V)\\
\gamma&\mapsto \Exp(\gamma)\\
s(V)&\leftarrow s_{++}(V)\oc\log\\
\log(\gamma)&\mapsfrom \gamma.
\end{align*}
\end{lemma}

\begin{lemma}\label{themaprho} We have a well-defined continuous map
\begin{align*}
\rho\co \Hom(V,W)&\to s_{+}(V)\\
\gamma&\mapsto(\gamma^{\dag}\gamma)^{\frac{1}{2}}
\end{align*}
with $\IM(\rho(\gamma))=(\Ker(\gamma))^{\bot}$.
\end{lemma}

\begin{lemma}\label{themapsigma} For each $\gamma\in\Hom(V,W)$ there is a well-defined continuous map
\[
\sigma(\gamma)\co (\Ker(\gamma))^{\bot}\to W,\phantom{xx} \sigma(\gamma):=\gamma\circ\rho(\gamma)^{-1}.
\]
Moreover, $\sigma(\gamma)$ is a linear isometry and $\gamma=\sigma(\gamma)\circ\rho(\gamma)$.
\end{lemma}

\begin{lemma}\label{themaplambda} Let $f\co \Real\to \Real^{+}$ be given by $f(x):=\max(x,0)$. Then we have a well-defined continuous map
\begin{align*}
\lambda_{k}\co s(V)&\to s_{+}(V)\\
\alpha&\mapsto f(\alpha-e_{d_{0}-k-1}(\alpha)).
\end{align*}
\end{lemma}

\begin{lemma}\label{thetaEalpha} $s(V)\times\mathcal{L}(V,W)\cong\inj(V,W)$ via
\begin{align*}
\kappa'\co s(V)\times\mathcal{L}(V,W)&\to\inj(V,W)\\
(\alpha,\theta)&\mapsto -\theta\circ \Exp(\alpha)\\
(\log(\rho(\gamma)),-\sigma(\gamma))&\mapsfrom \gamma.
\end{align*}
We thus have a continuous extension $\kappa\co S^{s(V)}\wedge\mathcal{L}(V,W)_{\infty}\cong\inj(V,W)_{\infty}$ and collapse map $\kappa^{!}\co S^{\Hom(V,W)}\to S^{s(V)}\wedge\mathcal{L}(V,W)_{\infty}$.
\end{lemma}

The starting point for our functional calculus variation are the below spaces; these model spaces of eigenvalues of self-adjoint endomorphisms.

\begin{definition}\label{TheSpacesD} We define the following spaces for $d\geqslant 1$ and $0\leqslant i\leqslant d-2$.
\begin{itemize}
\item $D'(d):=\{(t_{0},\ldots,t_{d-1})\in\Real^{d}:t_{0}\leqslant\ldots\leqslant t_{d-1}\}$.
\item $D(d):=(D'(d))_{\infty}$.
\item $F_{i}(D'(d)):=\{(t_{0},\ldots,t_{d-1})\in D'(d):t_{i}=t_{i+1}\}$.
\item $F_{i}(D(d)):=(F_{i}(D'(d)))_{\infty}$.
\item $D_{+}'(d):=\{(t_{0},\ldots,t_{d-1})\in D'(d):t_{0}\geqslant 0\}$.
\item $D_{+}(d):=(D_{+}'(d))_{\infty}$.
\item $F_{i}(D'_{+}(d)):=\{(t_{0},\ldots,t_{d-1})\in D'_{+}(d):t_{i}=t_{i+1}\}$.
\item $F_{i}(D_{+}(d)):=(F_{i}(D'_{+}(d)))_{\infty}$.
\item $D_{0}'(d):=\{(t_{0},\ldots,t_{d-1})\in D'(d):t_{0}=0\}$.
\item $D_{0}(d):=(F_{i}(D'(d)))_{\infty}$.
\end{itemize}
We refer to the $F_{i}(D(d))$ as the faces of $D(d)$, the faces of $D_{+}(d)$ are the spaces $F_{i}(D_{+}(d))$ and the space $D_{0}(d)$. We call $D(d)$ and $D_{+}(d)$ facial spaces and say that a self-map of $D(d)$ or $D_{+}(d)$ is facial if it preserves faces. Let $F(d)$ be the space of facial self-maps of $D(d)$ and let $F_{+}(d)$ be the space of facial self-maps of $D_{+}(d)$.

More generally, let $X$ and $Y$ be based spaces that have a notion of faces, so that we can talk about facial maps $X\to Y$. For example if $X=D(d)$ and $Y=D(d)\wedge Z$ for some $Z$ then $f\co X\to Y$ is facial if $f(F_{i}(D(d)))\subseteq F_{i}(D(d))\wedge Z$. Then we denote the space of facial maps from $X$ to $Y$ by $\FMap(X,Y)$. If $X=Y$ then we write $\FMap(X)$ for the space of facial self-maps of $X$. 
\end{definition}

The two technical lemmas we need to set up the machinery are easy to check. We fix $V$ to be a Hermitian space of dimension $d$.

\begin{lemma}\label{EigenvalueMap} Let $\eta'\co s(V)\to D'(d)$ be the eigenvalue map $\alpha\mapsto (e_{0}(\alpha),\ldots, e_{d-1}(\alpha))$. Then $\eta'$ is a continuous proper surjection and hence the map $\eta:=(\eta')_{\infty}\co S^{s(V)}\to D(d)$ is a quotient map.
\end{lemma}

\begin{lemma}\label{DiagonalMatrixMap} For $t\in \Real^{d}$ set $\Delta(d)$ to be the diagonal matrix with entries $t$. Define
\begin{align*}
\nu'\co \mathcal{L}(\Complex^{d},V)\times D'(d)&\to s(V)\\
(\alpha,t)&\mapsto \alpha\Delta(t)\alpha^{\dag}.
\end{align*}
Then $\nu'$ is a continuous proper surjection and hence $\nu:=(\nu')_{\infty}$ is a quotient map.
\end{lemma}

\begin{proposition}\label{fnalvarA} Let $X$ be a based space and let $f\co D(d)\to  D(d)\wedge X$ be facial. Then there exists a unique map $\mathfrak{A}_{f}\co s(V)_{\infty}\to s(V)_{\infty}\wedge X$ making
\[
\xymatrix@C=1.5cm{\mathcal{L}(\Complex^{d},V)_{\infty}\wedge D(d)\ar[d]_{1\wedge f}\ar[r]^{\phantom{xxxx}\nu}&s(V)_{\infty}\ar[d]^{\mathfrak{A}_{f}}\ar[r]^{\eta}&D(d)\ar[d]^{f}\\
\mathcal{L}(\Complex^{d},V)_{\infty}\wedge D(d)\wedge
X\ar[r]_{\phantom{xxxx}\nu\wedge1}& s(V)_{\infty}\wedge
X\ar[r]_{\eta\wedge 1}&D(d)\wedge X}
\]
commute; moreover, the associated map
\begin{align*}
\mathfrak{A}\co \FMap (D(d),D(d)\wedge X)&\to \Map (s(V)_{\infty},s(V)_{\infty}\wedge X)\\
f&\mapsto \mathfrak{A}_{f}
\end{align*}
is continuous. Furthermore, we have an explicit description of $\mathfrak{A}_{f}(\alpha)$. Choose an orthonormal basis of
eigenvectors $v_{0},\ldots,v_{d-1}$ of $\alpha$ with eigenvalues $e_{0}\leqslant\ldots\leqslant e_{d-1}$; then if
$f(e_{0},\ldots,e_{d-1})=(s_{0},\ldots,s_{d-1})\wedge x$ we have $\mathfrak{A}_{f}(\alpha)=f(\alpha)\wedge x$ where $f(\alpha)$ is the endomorphism with eigenvectors $v_{i}$ and eigenvalues $s_{i}$.
\end{proposition}
\begin{proof} We first need to check that if $\nu(\alpha,t)=\nu(\alpha',t')$ then $\nu(\alpha,f(t))=\nu(\alpha',f(t'))$, but this follows from the fact that if $f(t)=s\wedge x$ then the centralizer of $\Delta(s)$ is contained within the centralizer of $\Delta(t)$. The described map $\mathfrak{A}_{f}$ clearly fits into the square and moreover it is unique because $\nu$ is surjective. The map $\mathfrak{A}_{f}\circ\nu=\nu\circ(1\wedge f)$ is continuous and so it follows that $\mathfrak{A}_{f}$ is continuous as $\nu$ is a quotient.

We have an adjunction 
\[
\Map(\FMap (D(d),D(d)\wedge X), \Map (s(V)_{\infty},s(V)_{\infty}\wedge X))\cong\Map(\FMap (D(d),D(d)\wedge X)\wedge s(V)_{\infty},s(V)_{\infty}\wedge X).
\]
Hence if we show that the adjoint $\mathfrak{A}^{\#}\co \FMap (D(d),D(d)\wedge X)\wedge s(V)_{\infty}\to s(V)_{\infty}\wedge X$ is continuous then continuity of $\mathfrak{A}$ follows. Let $\text{eval}$ be given by
\begin{align*}
\text{eval}\co \FMap (D(d),D(d)\wedge X)\wedge \mathcal{L}(\Complex^{d},V)_{\infty}\wedge D(d)&\to \mathcal{L}(\Complex^{d},V)_{\infty}\wedge D(d)\wedge X\\
(f,\alpha,t)&\mapsto (\alpha,f(t)).
\end{align*}
We have a commutative diagram
\[
\xymatrix{\FMap (D(d),D(d)\wedge X)\wedge \mathcal{L}(\Complex^{d},V)_{\infty}\wedge D(d)\ar[r]^{\phantom{xxxxxxxx}\text{eval}}\ar[d]_{1\wedge\nu}&\mathcal{L}(\Complex^{d},V)_{\infty}\wedge D(d)\wedge X\ar[d]^{\nu\wedge 1}\\
\FMap (D(d),D(d)\wedge X)\wedge s(V)_{\infty}\ar[r]_{\phantom{xxxxxxxx}\mathfrak{A}^{\#}}&s(V)_{\infty}\wedge X}
\]
and hence $\mathfrak{A}^{\#}\circ (1\wedge \nu)=(\nu\wedge 1)\circ\text{eval}$ is continuous. The map $(1\wedge \nu)$ is a quotient, thus $\mathfrak{A}^{\#}$ and $\mathfrak{A}$ are continuous.
\end{proof}

It is easy to see that the above holds for spaces of non-negative selfadjoint endomorphisms.

\begin{corollary}\label{fnvarApositive} The maps $\eta'$ and $\nu'$ restrict to
\[
\eta'\co s_{+}(V)\to D'_{+}(d)
\]
\[
\nu'\co \mathcal{L}(\Complex^{d},V)\times D'_{+}(d)\to s_{+}(V).
\]
Let $f\co D_{+}(d)\to D_{+}(d)\wedge X$ be facial. Then there exists a unique map $\mathfrak{A}_{f}$ holding the properties stated in Proposition~\ref{fnalvarA} and making
\[
\xymatrix@C=1.5cm{\mathcal{L}(\Complex^{d},V)_{\infty}\wedge D_{+}(d)\ar[d]_{1\wedge f}\ar[r]^{\phantom{xxxx}\nu}&s_{+}(V)_{\infty}\ar[d]^{\mathfrak{A}_{f}}\ar[r]^{\eta}&D_{+}(d)\ar[d]^{f}\\
\mathcal{L}(\Complex^{d},V)_{\infty}\wedge D_{+}(d)\wedge
X\ar[r]_{\phantom{xxxx}\nu\wedge1}& s_{+}(V)_{\infty}\wedge
X\ar[r]_{\eta\wedge 1}&D_{+}(d)\wedge X}
\]
commute.
\end{corollary}

This result can be extended to build self-maps of $S^{\Hom(V,W)}$ for $V$ and $W$ Hermitian, $V$ of dimension $d$ and $W$ such that $\Dim(W)\geqslant d$. We again need two technical lemmas to set up the machinery, the proofs are easy to check.

\begin{lemma}\label{RhoIsProper} The map $\rho\co \Hom(V,W)\to s_{+}(V)$ is a proper surjection. Hence the based extension $\rho_{\infty}$ is a quotient map. Abusing notation we also denote this extension by $\rho$.
\end{lemma}

\begin{lemma}\label{TheMapMu} Define
\begin{align*}
\mu'\co s_{+}(V)\times\mathcal{L}(V,W)&\to \Hom(V,W)\\
(\alpha,\theta)&\mapsto -\theta\circ\alpha
\end{align*}
Then $\mu'$ is a continuous proper surjection and hence $\mu:=(\mu')_{\infty}$ is a quotient map.
\end{lemma}

\begin{proposition}\label{fnalvarB} Let $X$ be a based space and let $f\co D_{+}(d)\to D_{+}(d)\wedge X$ be facial. Then there exists a unique map $\mathfrak{B}_{f}\co S^{\Hom(V,W)}\to  S^{\Hom(V,W)}\wedge X$ making
\[
\xymatrix@C=1.5cm{\mathcal{L}(V,W)_{\infty}\wedge s_{+}(V)_{\infty}\ar[d]_{1\times \mathfrak{A}_{f}}\ar[r]^{\phantom{xxxx}\mu}&S^{\Hom(V,W)}\ar[d]^{\mathfrak{B}_{f}}\ar[r]^{\rho}&s_{+}(V)_{\infty}\ar[d]^{\mathfrak{A}_{f}}\\
\mathcal{L}(V,W)_{\infty}\wedge s_{+}(V)_{\infty}\wedge
X\ar[r]_{\phantom{xxxx}\mu\wedge1}& S^{\Hom(V,W)}\wedge
X\ar[r]_{\rho\wedge 1}&s_{+}(V)_{\infty}\wedge X}
\]
commute; moreover, the associated map 
\begin{align*}
\mathfrak{B}\co F(\mathfrak{A})&\to \Map(S^{\Hom(V,W)},S^{\Hom(V,W)}\wedge X)\\
f&\mapsto \mathfrak{B}_{f}
\end{align*}
is continuous. Here $F(\mathfrak{A})$ is the space of all maps $s_{+}(V)_{\infty}\to s_{+}(V)_{\infty}\wedge X$ of the form $\mathfrak{A}_{f}$. Furthermore, we have an explicit description of $\mathfrak{B}_{f}(\gamma)$. Choose an orthonormal basis of eigenvectors $v_{0},\ldots,v_{d-1}$ of $\gamma^{\dag}\gamma$ with eigenvalues $e_{0}^{2}\leqslant\ldots\leqslant e_{d-1}^{2}$ such that $\gamma(v_{i})=e_{i}m_{i}$ for some $m_{i}$ orthonormal in $W$. Then if $f(e_{0},\ldots,e_{d-1})=(s_{0},\ldots,s_{d-1})\wedge x$ we have $\mathfrak{B}_{f}(\gamma)=f(\gamma)\wedge x$, where $f(\gamma)$ is the homomorphism sending each $v_{i}$ to $s_{i}m_{i}$.
\end{proposition}
\begin{proof} Let $\mu(\theta,\alpha)=\mu(\theta',\alpha')$, that $\mu(\theta,\mathfrak{A}_{f}(\alpha))=\mu(\theta',\mathfrak{A}_{f}(\alpha'))$ follows from the fact that $\Ker(\alpha)\subseteq\Ker(\mathfrak{A}_{f}(\alpha))$. The described map $\mathfrak{B}_{f}$ makes the diagram commute and moreover this map is unique as $\mu$ is surjective. As in the proof of \ref{fnalvarA} $\mathfrak{B}_{f}$ is continuous as $\mu$ is a quotient map.

We again rely on an adjunction argument to show continuity of $\mathfrak{B}$; we show that the adjoint $\mathfrak{B}^{\#}\co F(\mathfrak{A})\wedge S^{\Hom(V,W)}\to S^{\Hom(V,W)}\wedge X$ is continuous. Let $\text{eval}$ be defined by
\begin{align*}
\text{eval}\co F(\mathfrak{A})\wedge\mathcal{L}(V,W)_{\infty}\wedge s_{+}(V)_{\infty}&\to\mathcal{L}(V,W)_{\infty}\wedge s_{+}(V)_{\infty}\wedge X\\
(\mathfrak{A}_{f},\theta,\alpha)&\mapsto (\theta,\mathfrak{A}_{f}(\alpha)).
\end{align*}
We have a commutative diagram
\[
\xymatrix{F(\mathfrak{A})\wedge\mathcal{L}(V,W)_{\infty}\wedge s_{+}(V)_{\infty}\ar[r]^{\text{eval}}\ar[d]_{1\wedge\mu}&\mathcal{L}(V,W)_{\infty}\wedge s_{+}(V)_{\infty}\wedge X\ar[d]^{\mu\wedge 1}\\
F(\mathfrak{A})\wedge S^{\Hom(V,W)}\ar[r]_{\mathfrak{B}^{\#}}&S^{\Hom(V,W)}\wedge X}
\]
and hence $\mathfrak{B}^{\#}\circ (1\wedge\mu)$ is continuous. The map $(1\wedge\mu)$ is a quotient, thus $\mathfrak{B}^{\#}$ and $\mathfrak{B}$ are continuous.
\end{proof}

\subsection{Building a cofibre sequence using the functional calculus}\label{Building a cofibre sequence using the functional calculus}

We now give a concrete example of this functional calculus by building an NDR (Neighbourhood Deformation Retract) pair, which we use throughout the rest of the document. We take our definition of NDR as follows.

\begin{definition}\label{NDR} Let $X$ be a space and $A$ a closed subspace. We say that a pair of continuous maps $(u\co X\to [0,1],h\co [0,1]\times X\to X)$ represents $(X,A)$ as an NDR pair if:
\begin{enumerate}
\item $h_{1}(x)=x$ for all $x\in X$.
\item $h_{t}(a)=a$ for all $t\in[0,1]$ and $a\in A$.
\item $h_{0}(x)\in A$ for all $x\in X$ such that $u(x)<1$.
\item $u^{-1}(0)=A$.
\end{enumerate}
\end{definition}

The next three lemmas have routine proofs.

\begin{lemma}\label{unitdiscNDR} Let $X$ be the upper half disc $\{z\in\Complex:|z|\leqslant 1, \IM(z)\geqslant 0\}$ and let $Y$ be  the upper semicircle $\{z\in X:|z|=1\}$ with basepoint $z=-1$:

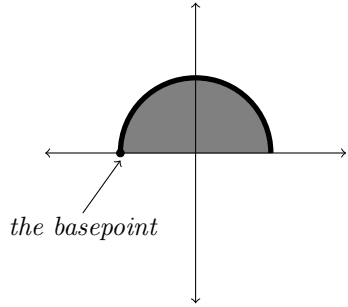
\begin{center}
\begin{tikzpicture}
\filldraw[color=gray] (1,0) arc (0:180:1) -- cycle;
\draw[<->](-2,0) -- (2,0);
\draw[<->](0,-2) -- (0,2);
\draw[line width=2pt](1,0) arc (0:180:1);
\filldraw[black](-1,0) circle (0.5mm);
\draw(-1.5,-1) node {the basepoint};
\draw[->](-1.5,-0.8) -- (-1,-0.1);
\end{tikzpicture}
\end{center}

Then
\begin{align*}
u''(re^{i\theta})&:=\min(1,2-2r)\\
h_{t}''(re^{i\theta})&:=\min(1,(2-t)r)e^{i\theta}.
\end{align*}
make $(X,Y)$ into an NDR pair.
\end{lemma}

\begin{lemma}\label{theconformalmapphi} There is a relative homeomorphism $\phi\co(D_{+}(2),D_{0}(2))\cong(X,Y)$ given by:
\begin{align*}
\phi\co D_{+}(2)&\to X\\
(t_{0},t_{1})&\mapsto\frac{i-(t_{1}+it_{0})^{2}}{i+(t_{1}+it_{0})^{2}}
\end{align*}

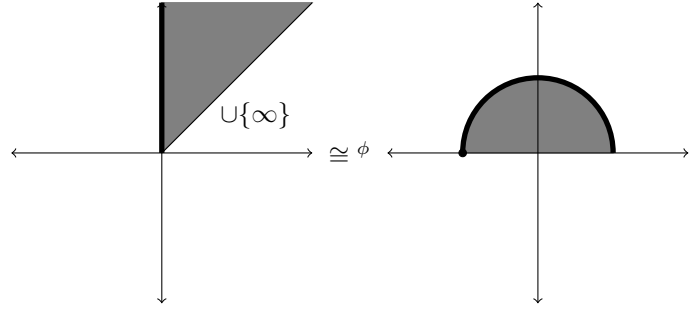
\begin{center}
\begin{tikzpicture}
\filldraw[color=gray] (0,0) -- (2,2) -- (0,2) -- cycle;
\draw[<->](-2,0) -- (2,0);
\draw[<->](0,-2) -- (0,2);
\draw[line width=2pt](0,0) -- (0,2);
\draw(0,0) -- (2,2);
\draw(1.25,0.5) node {$\cup\{\infty\}$};
\draw(2.5,0) node{$\cong{^\phi}$};
\filldraw[color=gray] (6,0) arc (0:180:1) -- cycle;
\draw[<->](3,0) -- (7,0);
\draw[<->](5,-2) -- (5,2);
\draw[line width=2pt](6,0) arc (0:180:1);
\filldraw[black](4,0) circle (0.5mm);
\end{tikzpicture}
\end{center}
\end{lemma}

\begin{lemma}\label{D2NDRpair} $(D_{+}(2),D_{0}(2))$ is an NDR pair via
\begin{align*}
u'(t_{0},t_{1})&:=u''\circ\phi(t_{0},t_{1})\\
h_{t}'(t_{0},t_{1})&:=\phi^{-1}\circ h_{t}''\circ \phi(t_{0},t_{1}).
\end{align*}
\end{lemma}

We want to build a new NDR pair out of the pair of Lemma~\ref{D2NDRpair}. To do this we need one more construction. Let $f\co D_{+}(2)\to D_{+}(2)$ be facial. Then $f$ can be written in the form $f(t_{0},t_{1})=(g(t_{0},t_{1}),g(t_{0},t_{1})+h(t_{0},t_{1}))$ for functions $g$, $h\co D_{+}(2)\to \Real^{+}$ such that $h(t,t)=0$.

\begin{proposition}\label{TheHatConstruction} Define $\hat{f}\co D(d)\to D(d)$ by
\[
\hat{f}(t_{0},\ldots,t_{d-1})_{i}=\left\{\begin{array}{ll}g(t_{0},t_{d-1})+\frac{t_{i}-t_{0}}{t_{d-1}-t_{0}}h(t_{0},t_{d-1})&\text{if }t_{0}<t_{d-1}\\
g(t_{0},t_{d-1})&\text{if }t_{0}=t_{d-1}\end{array}\right.
\]
\[
\hat{f}(\infty)=\infty.
\]
Then $\hat{f}$ is a continuous facial map and the map 
\begin{align*}
\operatorname{hat}\co F_{+}(2)&\to F_{+}(d)\\
f&\mapsto \hat{f}
\end{align*}
is continuous.
\end{proposition}
\begin{proof} Most of the claims are easy to show, though continuity of $\hat{f}$ requires a limit argument. The only real issue is checking that the map $\operatorname{hat}$ is continuous, which relies on another adjunction argument similar to those used in \ref{fnalvarA} and \ref{fnalvarB}. Recall that $F_{+}(d)$ is the space of facial self-maps of $D_{+}(d)$. We have an adjunction
\[
\Map(F_{+}(2),\Map(D_{+}(d),D_{+}(d))) \cong \Map(F_{+}(2)\wedge D_{+}(d),D_{+}(d)).
\]
Observe that $F_{+}(d)\subset\Map(D_{+}(d),D_{+}(d))$, thus continuity of $\operatorname{hat}$ follows from the continuity of the adjoint $\operatorname{hat}^{\#}$. Let $\Delta_{d-2}$ be the standard $(d-2)$--simplex which we take to be parameterized by $d-3$ increasing coordinates in $[0,1]$. Define
\begin{align*}
\lambda'\co D_{+}'(2)\times \Delta_{d-2}&\to D_{+}'(d)\\
(t_{0},t_{1},s_{0},\ldots,s_{d-3})&\mapsto (t_{0},t_{0}+s_{0}(t_{1}-t_{0}),\ldots,t_{0}+s_{d-3}(t_{1}-t_{0}),t_{1}).
\end{align*}
The map $\lambda'$ is a proper surjection, hence $\lambda:=(\lambda')_{\infty}$ is a quotient. Let $\text{eval}$ be the map
\begin{align*}
\text{eval}\co F_{+}(2)\wedge D_{+}(2)\wedge (\Delta_{d-2})_{\infty}&\to D_{+}(2)\wedge (\Delta_{d-2})_{\infty}\\
(f,t,s)&\mapsto (f(t),s).
\end{align*}
We have a commutative diagram
\[
\xymatrix{F_{+}(2)\wedge D_{+}(2)\wedge (\Delta_{d-2})_{\infty}\ar[r]^{\phantom{xx}\text{eval}}\ar[d]_{(1\wedge \lambda)}&D_{+}(2)\wedge (\Delta_{d-2})_{\infty}\ar[d]^{\lambda}\\
F_{+}(2)\wedge D_{+}(d)\ar[r]_{{\phantom{xx}\operatorname{hat}^{\#}}}&D_{+}(d)}
\]
and hence $\text{hat}^{\#}\circ (1\wedge \lambda)$ is continuous; the continuity of $\text{hat}$ follows.
\end{proof}

This construction is used to build the below NDR pair, the proof is simple to check.

\begin{proposition}\label{homNDRpair} $(S^{\Hom(V,W)},\inj(V,W)^{c}_{\infty})$ is an NDR pair via
\begin{align*}
u(\gamma)&:=u'(e_{0}(\rho(\gamma)),e_{d-1}(\rho(\gamma)))\\
h_{t}(\gamma)&:=\mathfrak{B}_{\widehat{h'_{t}}}(\gamma).
\end{align*}
\end{proposition}

It is standard that one can build a cofibre sequence from an NDR pair. In our case an NDR pair $(X,A)$ produces a cofibre sequence
\[
A\overset{i}{\to} X\overset{p}{\to} \frac{X}{A}\overset{e}{\to}\Sigma A
\]
where $i$ is the inclusion, $p$ the collapse and $e$ the composition $X/A\overset{r}{\to} C_{i}\overset{d}{\to}\Sigma A$ with $d$ the standard collapse and $r$ the map
\begin{align*}
r\co \frac{X}{A}&\to C_{i}\\
x&\mapsto (u(x),h_{0}(x)).
\end{align*}
The below result then follows.

\begin{corollary}\label{HomCofibSeq} We have a cofibre sequence
\[
\inj(V,W)^{c}_{\infty}\overset{i}{\to}S^{\Hom(V,W)}\overset{p}{\to}\inj(V,W)_{\infty}\overset{e}{\to}\Sigma\inj(V,W)^{c}_{\infty}.
\]
\end{corollary}

The proof of Theorem~\ref{IntroTower} relies on showing that many sequences are isomorphic to modifications of this sequence.

\subsection{Homotopy classification in the functional calculus}\label{Homotopy classification in the functional calculus}

The strength of the extended functional calculus is its homotopy properties. It is possible to determine a homotopy classification of maps of the form $\mathfrak{A}_{f}$ or $\mathfrak{B}_{f}$, we achieve this classification by proving Theorem~\ref{IntroHomotopy}. Let $f$, $g\co D(d)\to D(d)$ be facial and such that $f\simeq g$ through a facial homotopy $h_{t}$. Then $\mathfrak{A}_{f}\simeq \mathfrak{A}_{g}$ via $\mathfrak{A}_{h_{t}}$ and $\mathfrak{B}_{f}\simeq \mathfrak{B}_{g}$ via $\mathfrak{B}_{h_{t}}$. Hence we study the homotopy type of facial self-maps of $D(d)$.

\begin{definition}\label{faceintersectionBsigma} Let $\sigma\subseteq \{0,\ldots,d-2\}$. Then define $\bar{B}_{\sigma}$ to be the intersection of faces $\bigcap_{i\notin \sigma}F_{i}(D(d))$. Moreover define $\bar{B}[k]$ be the union of all $\bar{B}_{\sigma}$ with $|\sigma|\leqslant k$; note that $\bar{B}[0]=\bar{B}_{\emptyset}$ and $\bar{B}[d-1]=D(d)$. We say that a self-map of $\bar{B}[k]$ is facial if it preserves each $\bar{B}_{\sigma}$---this is consistent with the earlier definition of a facial map.
\end{definition}

We need two brief technical lemmas to proceed, recalling the notation $B^{n}$ for a ball of dimension $n$.

\begin{lemma}\label{extendingboundariesofballs} Suppose $X\cong B^{n+1}$ and $Y\cong B^{n}$, and let $p$ be a map $\partial X\to Y$. Then there exists an extension $\tilde{p}\co X\to Y$ of $p$.
\end{lemma}
\begin{proof} Without loss of generality we can assume that $X=B^{n+1}$ and $Y=B^{n}$. Parameterize $X$ by coordinates $(x,t)$ for $x$ a point on the boundary and $t$ a scalar. The extension is $\tilde{p}(x,t):=tp(x)$.
\end{proof}

\begin{lemma}\label{extendingballshomotopies} Suppose $Y\cong B^{n}$ and $f$, $g\co Y\to Y$ are maps such that $h\co [0,1]\times\partial Y\to Y$ gives a homotopy from $f|_{\partial Y}$ to $g|_{\partial Y}$. Then there exists an extension $\tilde{h}\co [0,1]\times Y\to Y$ that provides a homotopy between $f$ and $g$.
\end{lemma}
\begin{proof} Set $X:=[0,1]\times Y$, it trivially follows that $X\cong B^{n+1}$. Note $\partial X\cong([0,1]\times \partial Y)\cup(\{0,1\}\times Y)$, define $p\co\partial X\to Y$ to be $h$ on $[0,1]\times \partial Y$, $f$ on $\{0\}\times Y$ and $g$ on $\{1\}\times Y$. Use Lemma~\ref{extendingboundariesofballs} to extend $p$ to the required homotopy $\tilde{h}$.
\end{proof}

We can now prove the key lemma.

\begin{lemma}\label{facialhomotopiesinduction} Let $f$, $g\co D(d)\to D(d)$ be facial and such that $f|_{\bar{B}[k]}\simeq g|_{\bar{B}[k]}$ through facial maps. Then this homotopy can be extended to a facial homotopy $f|_{\bar{B}[k+1]}\simeq g|_{\bar{B}[k+1]}$.
\end{lemma}
\begin{proof} Let $h_{k}$ be the homotopy $[0,1]\times \bar{B}[k]\to \bar{B}[k]$ agreeing with $f$ on $0$ and $g$ on $1$. Now let $\bar{B}_{\sigma}$ be such that $|\sigma|=k+1$. Then we have $\bar{B}_{\sigma}\cong B^{k+2}\subset\bar{B}[k+1]$ and $\partial \bar{B}_{\sigma}\subset \bar{B}[k]$. Restrict $f$ and $g$ to $f|_{\bar{B}_{\sigma}}$ and $g|_{\bar{B}_{\sigma}}$ and restrict $h_{k}$ to a homotopy $f|_{\partial \bar{B}_{\sigma}}\simeq g|_{\partial \bar{B}_{\sigma}}$. This extends to give a homotopy $h_{k+1,\sigma}\co[0,1]\times \bar{B}_{\sigma}\to \bar{B}_{\sigma}$ via Lemma~\ref{extendingballshomotopies} which agrees with $h_{k}$ on the boundary, $f|_{\bar{B}_{\sigma}}$ on $0$ and $g|_{\bar{B}_{\sigma}}$ on $1$. Hence we have a family of maps $\{h_{k+1,\sigma}\}_{|\sigma|=k+1}$. If $\sigma\neq\tau$ observe that $\bar{B}_{\sigma}\cap \bar{B}_{\tau}\subset \bar{B}[k]$. Thus the two homotopies $h_{k+1,\sigma}$ and $h_{k+1,\tau}$ agree on the intersection as they are both $h_{k}$ on $\bar{B}[k]$. Patch the family together to get a homotopy $h_{k+1}\co[0,1]\times \bar{B}[k+1]\to \bar{B}[k+1]$ extending $h_{k}$ and giving $f\simeq g$. That this homotopy is facial is trivial to observe.
\end{proof}

We now observe that $\bar{B}_{\emptyset}\cong S^{1}$. Hence the following theorem, a restatement of Theorem~\ref{IntroHomotopy}, follows by induction using Lemma~\ref{facialhomotopiesinduction}.

\begin{theorem}\label{facialhomotopiesviadegrees} Let $f$, $g\co D(d)\to D(d)$ be facial and such that $f$ and $g$ have the same
degree on $\bar{B}_{\emptyset}$. Then $f\simeq g$ through facial maps.
\end{theorem}

Hence we now have a criterion for saying whether two maps $\mathfrak{A}_{f}$ and $\mathfrak{A}_{g}$ are homotopic: we have induced maps $f'$, $g'\co S^{1}\to S^{1}$ given by $f'(t):=f(t,\ldots,t)$ and $g'(t):=g(t,\ldots,t)$ and if $f'$ and $g'$ have the same degree then $\mathfrak{A}_{f}\simeq \mathfrak{A}_{g}$.

\section{An equivariant stable tower over isometries}\label{An equivariant stable tower over isometries} 

Recall the setup of $G$, $V_{0}$ and $V_{1}$ discussed in the introduction. We spend this section proving the following theorem, a more detailed technical statement of Theorem~\ref{IntroTower} which serves as the main result of the paper.

\begin{theorem}\label{TheMainResult} There is a natural tower of spectra
\[
\Ell_{\infty}\overset{\pi_{d_{0}}}{\to} X_{d_{0}-1}\overset{\pi_{d_{0}-1}}{\to}\ldots\overset{\pi_{2}}{\to} X_{1}\overset{\pi_{1}}{\to} S^{0}
\]
such that:
\begin{enumerate}
\item The map $\Ell_{\infty}\to S^{0}$ is the compactified version of the projection map $\Ell\to\text{pt}$.
\item The stable cofibre of the map $\pi_{k}\co X_{k}\to X_{k-1}$ is $G_{k}(V_{0})^{\Real\oplus \Hom(T,V_{1}-V_{0})\oplus s(T)}$, i.e.\ the triangle
\[
\xymatrix@C=1.5cm
{X_{k}\ar[d]_{\pi_{k}}&G_{k}(V_{0})^{\Hom(T,V_{1} - V_{0})\oplus s(T)}\ar[l]_{\phi_{k}\phantom{xxxxxxx}}\\
X_{k-1}\ar[ur]|\bigcirc_{\phantom{xxx}\delta_{k}}&}
\]
is a cofibre triangle.
\end{enumerate}
\end{theorem}

We first define, topologize and equip with group actions all the spectra in the tower. We do this mostly unstably. Recall that if $\alpha$ is self-adjoint then there is an inherent ordering on the eigenvalues $e_{j}(\alpha)$ and hence the eigenspaces $\Ker(\alpha-e_{j}(\alpha))$.

\begin{definition}\label{pkdefn} Let $P_{k}(\alpha)$ be the following subspace of $V_{0}$
\[
P_{k}(\alpha):=\left[\bigoplus_{j+k<d_{0}}(\Ker(\alpha-e_{j}(\alpha)))\right]^{\bot}.
\]
\end{definition}

\begin{definition}\label{Xkassets} Define the set
\[
\tilde{X}_{k}':=\{(\alpha,\theta):\alpha\in s(V_{0}), \theta\in \mathcal{L}(P_{k}(\alpha),V_{1})\}.
\]
\end{definition}

We now topologize this space.

\begin{definition}\label{Xkasspaces} There is a surjection 
\begin{align*}
\svo\times\Ell&\to \tilde{X}_{k}'\\
(\alpha,\theta)&\mapsto (\alpha,\theta|_{P_{k}(\alpha)}).
\end{align*}
Hence equip $\tilde{X}_{k}'$ with the topology of a quotient of $\svo\times\Ell$.
\end{definition}

This topology is useful, but later continuity arguments will be eased by an equivalent topology. The following lemma is simple to check.

\begin{lemma}\label{Xkothertopology} Recall $\rho$ and $\lambda_{k}$ from \ref{themaprho} and \ref{themaplambda}. We have a bijection
\begin{align*}
\tilde{X_{k}}'&\to\{(\alpha,\beta):\alpha\in\svo, \beta\co V_{0}\to V_{1}, \rho(\beta)=\lambda_{k}(\alpha)\}\\
(\alpha,\theta)&\mapsto (\alpha,-\theta\circ \lambda_{k}(\alpha)).
\end{align*}
We can topologize $\tilde{X}_{k}'$ as a subspace of $\svo\times \aich$, whereupon this bijection becomes a homeomorphism.
\end{lemma}

A quick check using \ref{quotientandsubspacearethesame} shows that these two topologies are the same.

\begin{definition}\label{mapspik} Define maps $\tilde{\pi}_{k}'\co\tilde{X}_{k}'\to \tilde{X}_{k-1}'$ by
\[
(\alpha,\theta)\mapsto(\alpha,\theta|_{P_{k-1}(\alpha)}).
\]
\end{definition}

The strictly commutative diagram
\[
\xymatrix{\svo\times\Ell\ar@{-}[d]_{1}\ar@{->>}[r]&\tilde{X}_{k}'\ar[d]^{\tilde{\pi}_{k}'}\\
\svo\times\Ell\ar@{->>}[r]&\tilde{X}_{k-1}'}
\]
proves that this map is well-defined, continuous and proper. Hence we define $\tilde{X}_{k}:=(\tilde{X}_{k}')_{\infty}$ and $\tilde{\pi}_{k}:=(\tilde{\pi}'_{k})_{\infty}$.

\begin{definition}\label{thespectraXk} Define the spectra $X_{k}:=S^{-\svo}\wedge\Sigma^{\infty} \tilde{X}_{k}$ and maps $\pi_{k}:=\Sigma^{-\svo}\Sigma^{\infty}\tilde{\pi}_{k}$.
\end{definition}

The unstable tower
\[
\tilde{X}_{d_{0}}=S^{\svo}\wedge\Ell_{\infty}\to\ldots\to \tilde{X}_{0}=S^{\svo}
\]
induces a stable tower
\[
X_{d_{0}}=\Ell_{\infty}\to\ldots\to X_{0}=S^{0}.
\]
It is clear that the map $\Ell_{\infty}\to S^{0}$ comes from the projection. Regarding equivariance, we recall that for any representations $V$ and $W$ the space $s(V)$ has an action by conjugation and that $\mathcal{L}(V,W)$ has an action by conjugation.

\begin{definition}\label{actiononXk} Equip $\tilde{X}_{k}'$ with the action $g.(\alpha,\theta):=(g.\alpha,g.\theta)$.
\end{definition}

Again, it's standard to check that this action is well-defined, is compatible with the topologies, makes the $\tilde{X}_{k}'$ into $G$--spaces and makes the $\tilde{\pi}_{k}'$ into $G$--maps. The $G$--spectra $X_{k}$ inherit their action from $\tilde{X}_{k}'$. Part $1$ of Theorem~\ref{TheMainResult} follows.

We also mention the topology on what we claim are the cofibres.

\begin{definition}\label{Zk} Define
\[
\tilde{Z}_{k}:=\{(W,\gamma,\psi):W\in G_{k}(V_{0}),\gamma\in\Hom(W,V_{1}),\psi\in s(W^{\bot})\}.
\]
\end{definition}

We topologize this via the following lemma. We first topologize $G_{k}(V_{0})$ as homeomorphic to 
\[
G'_{k}(V_{0}):=\{\pi\in \svo:\pi^{2}=\pi, \operatorname{trace}(\pi)=k\},
\]
this is a compact subset of $\svo$. 

\begin{lemma}\label{topologizingthefibrebundles} We have a bijection between $\tilde{Z}_{k}$ and the space
\[
\{(\pi,\beta,\xi):\pi\in G_{k}'(V_{0}),\beta\in\aich, \xi\in\svo, \beta\circ (1-\pi)=0,\xi\circ\pi=0\}
\]
given by 
\begin{align*}
(W,\gamma,\psi)&\mapsto (1_{W}\oplus 0_{W^{\bot}},\gamma\circ(1_{W}\oplus 0_{W^{\bot}}),\psi\circ(1_{W^{\bot}}\oplus 0_{W}))\\
(\IM(\pi),\beta|_{\IM(\pi)},\xi|_{\IM(1_{V_{0}}-\pi)})&\mapsfrom (\pi,\beta,\xi).
\end{align*}
Thus $\tilde{Z}_{k}$ is a subspace of $G_{k}'(V_{0})\times \aich\times\svo$. Moreover we have a surjection 
\begin{align*}
\mathcal{L}(\Complex^{k}\oplus \Complex^{d_{0}-k},V_{0})\times\Hom(\Complex^{k},V_{1})\times s(\Complex^{d_{0}-k})&\to \tilde{Z}_{k}\\
((\zeta,\eta),\gamma_{0},\psi_{0})&\mapsto (\IM(\zeta),\gamma_{0}\circ \zeta^{\dag},\eta\circ\psi_{0}\circ\eta^{\dag}).
\end{align*}
Hence $\tilde{Z}_{k}$ can also be topologized as a quotient. These two topologies are the same.
\end{lemma}

This result follows from an application of \ref{quotientandsubspacearethesame}. It is standard that $\tilde{Z}_{k}$ is a vector bundle over $G_{k}(V_{0})$.

\begin{definition}\label{EquivarianceOfBundles} Equip $G_{k}(V_{0})$ with the standard action $g.L:=g(L)$. Then equip $\tilde{Z}_{k}$ with the action $g.(W,\gamma,\psi):=(g.W,g.\gamma,g.\psi)$.
\end{definition}

There are things to check here, but it is an easy exercise to show that this action is well-defined, compatible with the topologies above and such that $\tilde{Z}_{k}$ is a $G$--vector bundle over $G_{k}(V_{0})$. Hence we can define the Thom space $G_{k}(V_{0})^{\Hom(T,V_{1})\oplus s(T^{\bot})}$, this is $(\tilde{Z}_{k})_{\infty}$ as $G_{k}(V_{0})$ is compact.

We now stabilize this bundle to make the claimed cofibre. Consider the $G$--spectrum
\[
S^{-\svo}\wedge\Sigma^{\infty}G_{k}(V_{0})^{\Hom(T,V_{1})\oplus s(T^{\bot})}.
\]
We identify it with the spectrum we claim is the cofibre via the following lemma.

\begin{lemma}\label{equibundleids} We have
\[
\Hom(T,V_{1} - V_{0})\oplus s(T)\oplus \svo\cong \Hom(T,V_{1})\oplus s(T^{\bot})
\]
as $G$--bundles over $G_{k}(V_{0})$.
\end{lemma}
\begin{proof} We first note a standard bundle identity, $\svo\cong s(T)\oplus s(T^{\bot})\oplus \Hom(T,T^{\bot})$. This follows from the decompositions of $\alpha\in \svo$ into $2\times 2$ matrices
\[
\alpha=\left(\begin{array}{cc}\beta& \delta^{\dag}\\
\delta & \gamma\end{array}\right)
\]
for $\beta\in s(W)$, $\gamma\in s(W^{\bot})$ and $\delta\in\Hom(W,W^{\bot})$ for each $W$ a fixed subspace of $V_{0}$. This identity is also equivariant. We also have another equivariant identity $\Hom(T,T)\cong 2.s(T)$, this follows from the classical decomposition of $\alpha\in\End(W)$ as $\alpha=\alpha_{0}+i\alpha_{1}$ with $\alpha_{0}$, $\alpha_{1}\in s(W)$. Combining these with the identity $\Hom(T,V_{0})\cong\Hom(T,T)\oplus \Hom(T,T^{\bot})$ leads to the result. 
\end{proof}

It follows that
\[
S^{-\svo}\wedge\Sigma^{\infty}G_{k}(V_{0})^{\Hom(T,V_{1})\oplus s(T^{\bot})}\cong G_{k}(V_{0})^{\Hom(T,V_{1}-V_{0})\oplus s(T)}.
\]

\subsection{The cofibre sequences}\label{The cofibre sequences}

We now prove part $2$ of Theorem~\ref{TheMainResult}. Recall the homeomorphism $\kappa\co \svo\wedge\Ell_{\infty}\cong\inj(V_{0},V_{1})_{\infty}$ from \ref{thetaEalpha}. The below claim is easy to check.

\begin{lemma}\label{themaptau} We have a homeomorphism
\begin{align*}
\tau\co\tilde{X}_{d_{0}-1}'&\overset{\cong}{\to} \Real\times\inj(V_{0},V_{1})^{c}\\
(\alpha,\theta)&\mapsto (e_{0}(\alpha),-\theta\circ(\alpha-e_{0}(\alpha)))
\end{align*}
and hence $\tilde{X}_{d_{0}-1}\cong\Sigma\inj(V_{0},V_{1})^{c}_{\infty}$.
\end{lemma}

Thus there is a unique map $\chi$ such that
\[
\xymatrix{\tilde{X}_{d_{0}}\ar[d]_{\tilde{\pi}_{d_{0}}}\ar[r]^{\kappa\phantom{xxx}}&\inj(V_{0},V_{1})_{\infty}\ar[d]^{\chi}\\
\tilde{X}_{d_{0}-1}\ar[r]_{\tau\phantom{xxx}}&\Sigma\inj(V_{0},V_{1})^{c}_{\infty}}
\] 
commutes. This map is observed to be given as follows, recalling $\rho$ and $\sigma$ from \ref{themaprho} and \ref{themapsigma}
\begin{align*}
\chi\co:\inj(V_{0},V_{1})_{\infty}&\to \Sigma\inj(V_{0},V_{1})^{c}_{\infty}\\
\gamma&\mapsto\left(\begin{array}{c}e_{0}(\log(\rho(\gamma)))\\
\sigma(\gamma)\circ(\log(\rho(\gamma))-e_{0}(\log(\rho(\gamma))))\end{array}\right).
\end{align*}
Recall the cofibre sequence
\[
\inj(V_{0},V_{1})_{\infty}\overset{e}{\to}\Sigma\inj(V_{0},V_{1})^{c}_{\infty}\overset{-\Sigma i}{\to} S^{\aich\oplus\Real}\overset{-\Sigma p}{\to}\Sigma \inj(V_{0},V_{1})_{\infty}
\]
from \ref{HomCofibSeq} and the construction $\mathfrak{B}$ from \ref{fnalvarB}. It is a standard fact that $\mathfrak{B}_{f}\co S^{\aich}\to S^{\Real\oplus\aich}$ factors through $\inj(V_{0},V_{1})_{\infty}{\to}\Sigma\inj(V_{0},V_{1})^{c}_{\infty}$ if and only if $f$ factors through $D_{+}(d_{0})/D_{0}(d_{0})\to \Sigma D_{0}(d_{0})$. Hence for the specific map $f$ given by
\begin{align*}
f\co \frac{D_{+}(d_{0})}{D_{0}(d_{0})}&\to \Sigma D_{0}(d_{0})\\
(t_{0},\ldots,t_{d_{0}-1})&\mapsto(u'(t_{0},t_{d_{0}-1}),\hat{h_{0}'}(t_{0},\ldots,t_{d_{0}-1}))
\end{align*}
we observe that $e=\mathfrak{B}_{f}$, recalling $u'$ and $\hat{h_{0}'}$ from the various constructions in Section~\ref{Building a cofibre sequence using the functional calculus}.

\begin{proposition}\label{homotopyinthetopofthetower} Let $g$ be the map
\begin{align*}
g\co\frac{D_{+}(d_{0})}{D_{0}(d_{0})}&\to \Sigma D_{0}(d_{0})\\
(t_{0},\ldots,t_{d_{0}-1})&\mapsto(\log(t_{0}),0,\log(t_{1})-\log(t_{0}),\ldots,\log(t_{d_{0}-1})-\log(t_{0})).
\end{align*}
Then $\chi=\mathfrak{B}_{g}$ and the map $g$ is homotopic through facial maps to the map $f$ defined above. Hence $e\simeq \chi$.
\end{proposition}
\begin{proof} The first claim is simple to verify. For the second, we note we have face-preserving homeomorphisms
\begin{align*}
\frac{D_{+}(d_{0})}{D_{0}(d_{0})}&\overset{\cong}{\to} D(d_{0})\\
t&\mapsto \log(t)
\end{align*}
and
\begin{align*}
\Sigma D_{0}(d_{0})&\overset{\cong}{\to} D(d_{0})\\
(s,t_{0}=0,\ldots,t_{d_{0}-1})&\mapsto (s+t_{0},\ldots,s+t_{d_{0}-1}).
\end{align*}
Hence we have maps $f'$, $g'\co D(d_{0})\to D(d_{0})$ induced by maps $f$ and $g$. We show that $f'$ and $g'$ are homotopic using Theorem~\ref{facialhomotopiesviadegrees}, this is enough to complete the proof. 

We have induced maps $f''$, $g''\co S^{1}\to S^{1}$ given by $f''(t)=f'(t,\ldots,t)$ and $g''(t)=g'(t,\ldots,t)$. We claim these have the same degree. It is easy to see that $g''$ is the identity and hence has degree $1$. The map $f''$ is explicitly given by
\[
t\mapsto\left\{\begin{array}{ll}
\log\left(\frac{8e^{t}}{1-6e^{t}}\right)& \quad t< -\log(6)\\
\infty& \quad\text{otherwise.}
\end{array}
\right.
\]
This can be checked by following through with all the definitions. We have a map
\begin{align*}
f'''\co\Real&\to\Real\\
t&\mapsto \log\left(\frac{e^{t}}{8+6e^{t}}\right).
\end{align*}
This is a strictly increasing embedding and it is easy to see that $f''$ is the collapse $(f''')^{!}$. Let $h_{s}\co\Real\to \Real$ be the homotopy $h_{s}(t)=st+(1-s)f'''(t)$. The map $(h_{s})^{!}$ provides a homotopy between $f''$ and the identity, hence $f''$ is degree $1$ and the proposition follows. 
\end{proof}

It follows that we have a cofibre sequence
\[
\inj(V_{0},V_{1})_{\infty}\overset{\chi}{\to}\Sigma\inj(V_{0},V_{1})^{c}_{\infty}\overset{-\Sigma i}{\to} S^{\aich\oplus\Real}\overset{-\Sigma p}{\to}\Sigma \inj(V_{0},V_{1})_{\infty}
\]
and hence we can build a cofibre sequence
\[
\tilde{X}_{d_{0}}\overset{\tilde{\pi}_{d_{0}}}{\to}\tilde{X}_{d_{0}-1}\to S^{\aich\oplus\Real}\to\Sigma \tilde{X}_{d_{0}}.
\]
Applying $\Sigma^{\infty}$ and smashing throughout by $S^{-\svo}$, it follows that we have a cofibre sequence
\[
X_{d_{0}}\overset{\pi_{d_{0}}}{\to}X_{d_{0}-1}\to G_{d_{0}}(V_{0})^{\Hom(T,V_{1} - V_{0})\oplus s(T)}\to \Sigma X_{d_{0}}
\]
building the top of the tower. The final thing to check is that all the material above is compatible with the stated $G$--actions but this is simple to observe.

Let $k<d_{0}$, we now prove that the cofibre of $\tilde{\pi}_{k}\co\tilde{X}_{k}\to \tilde{X}_{k-1}$ is $G_{k}(V_{0})^{\Real\oplus\Hom(T,V_{1})\oplus s(T^{\bot})}$. We first simplify by taking a quotient.

\begin{definition}\label{thequotients} Set $Y_{k}':=\{(\alpha,\theta)\in \tilde{X}_{k}':\dim(P_{k}(\alpha))< k\}$ and $Y_{k}:=(Y_{k}')_{\infty}$. Then abusing notation somewhat $Y_{k}'$ can be thought of as both a subspace of $\tilde{X}_{k}'$ and $\tilde{X}_{k-1}'$ and we have a map $\varpi_{k}\co\tilde{X}_{k}/Y_{k}\to \tilde{X}_{k-1}/Y_{k}$ induced from the map $\tilde{\pi}_{k}$.
\end{definition}

By \ref{quotientofacofibresequence} the cofibre of $\varpi_{k}$ is naturally equivalent to the cofibre of the map $\tilde{\pi}_{k}$. We need a single technical continuity argument before we can proceed.

\begin{lemma}\label{Pkiscont} Define $s_{k}(V_{0}):=\{\alpha\in \svo:\Dim(P_{k}(\alpha))=k\}$, then $P_{k}:s_{k}(V_{0})\to G_{k}(V_{0})$ is continuous. 
\end{lemma}
\begin{proof} Let $\svo^{\times}:=\{\alpha\in \svo:e_{i}(\alpha)\in\Real\backslash \{0\}\forall i\}$ and let $G(V_{0}):=\bigcup_{k} G_{k}(V_{0})$. Define $f\co\Real\backslash \{0\}\to \{0,1\}$ to be the function sending negative numbers to $0$ and positive numbers to $1$, then there is an associated continuous function $f\co\svo^{\times}\to G(V_{0})$. The claim is trivial if $k=0$ and if $k=d_{0}$ so now assume $0<k< d_{0}$, hence  $s_{k}(V_{0})=\{\alpha\in\svo:e_{d_{0}-k-1}(\alpha)<e_{d_{0}-k}(\alpha)\}$. Define a map $s_{k}(V_{0})\to \Real$ given by
\[
\alpha\mapsto 1/2(e_{d_{0}-k-1}(\alpha)+e_{d_{0}-k}(\alpha))
\]
and note that as $1/2(e_{d_{0}-k-1}(\alpha)+e_{d_{0}-k}(\alpha))$ is not an eigenvalue of $\alpha$ we have $\alpha-1/2(e_{d_{0}-k-1}(\alpha)+e_{d_{0}-k}(\alpha))\in s(V_{0})^{\times}$. The claim follows by observing that $P_{k}(\alpha)=f(\alpha-1/2(e_{d_{0}-k-1}(\alpha)+e_{d_{0}-k}(\alpha)))$.
\end{proof}

\begin{proposition}\label{thetaEalphaextended} Define
\[
\mathcal{I}_{k}':=\{(W,\gamma,\psi):W\in G_{k}(V_{0}),\gamma\in\inj(W,V_{1}),\psi\in s(W^{\bot})\}.
\]
Topologize $\mathcal{I}_{k}'$ as an open subspace of $\tilde{Z}_{k}$ from \ref{Zk} and define $\mathcal{I}_{k}:=(\mathcal{I}_{k}')_{\infty}$; note that this space has a $G$--action inherited from $\tilde{Z}_{k}$. Recall $\rho$ and $\sigma$ from \ref{themaprho} and \ref{themapsigma}, the maps $\mathfrak{q}_{k}\co\tilde{X}_{k}'\backslash Y_{k}'\to \mathcal{I}'_{k}$ and $\mathfrak{r}_{k}\co\mathcal{I}'_{k}\to \tilde{X}_{k}'\backslash Y_{k}'$ given by 
\[
\mathfrak{q}_{k}\co(\alpha,\theta)\mapsto\left(P_{k}(\alpha),-\theta\circ \Exp(\alpha|_{P_{k}(\alpha)}),-\log((e_{d_{0}-k}(\alpha)-\alpha)|_{P_{k}(\alpha)^{\bot}})\right)
\]
\[
\left((\log(e_{0}(\rho(\gamma)))-\Exp(-\psi))|_{W^{\bot}}\oplus \log(\rho(\gamma))|_{W},-\sigma(\gamma) \right)\mapsfrom(W,\gamma,\psi)\oc\mathfrak{r}_{k}
\]
are well-defined continuous $G$--maps that are inverses of each other, hence $\tilde{X}_{k}'\backslash Y_{k}'\cong \mathcal{I}'_{k}$ and thus $\tilde{X}_{k}/ Y_{k}\cong \mathcal{I}_{k}$.
\end{proposition}
\begin{proof} To prove this it needs to be checked that $\mathfrak{q}_{k}$ and $\mathfrak{r}_{k}$ are well-defined, that $\mathfrak{q}_{k}\circ\mathfrak{r}_{k}$ and $\mathfrak{r}_{k}\circ\mathfrak{q}_{k}$ are the identity and that $\mathfrak{q}_{k}$ and $\mathfrak{r}_{k}$ are continuous. These checks are standard, barring the continuity arguments. 

To check that $\mathfrak{q}_{k}$ is continuous, equip $\tilde{X}_{k}'\backslash Y_{k}$ with the quotient topology of \ref{Xkasspaces} and topologize $\mathcal{I}_{k}'$ as a subspace of $\tilde{Z}_{k}'$ and hence as a subspace of $G_{k}(V_{0})\times\aich\times\svo$ via \ref{topologizingthefibrebundles}. The continuity of $\mathfrak{q}_{k}$ then follows from continuity of the map
\begin{align*}
\{(\alpha,\theta)\in\svo\times\Ell:\Dim(P_{k}(\alpha))=k\}&\to G_{k}(V_{0})\times \aich\times\svo\\
(\alpha,\theta)&\mapsto(P_{k}(\alpha),-\theta\circ \Exp(\alpha),-\log(e_{d_{0}-k}(\alpha)-\alpha)).
\end{align*}
Continuity of this map follows from \ref{Pkiscont}. To demonstrate that $\mathfrak{r}_{k}$ is continuous firstly equip $\tilde{X}_{k}'\backslash Y_{k}$ with the subspace topology of \ref{Xkothertopology}. Next note that similar to \ref{topologizingthefibrebundles} we can equip $\mathcal{I}_{k}'$ with a quotient topology via the map
\begin{align*}
\mathcal{L}(\Complex^{k}\oplus \Complex^{d_{0}-k},V_{0})\times\inj(\Complex^{k},V_{1})\times s(\Complex^{d_{0}-k})&\to \mathcal{I}_{k}'\\
((\zeta,\eta),\gamma_{0},\psi_{0})&\mapsto (\IM(\zeta),\gamma_{0}\circ \zeta^{\dag},\eta\circ\psi_{0}\circ\eta^{\dag}).
\end{align*}
and note by \ref{quotientandsubspacearethesame} that this topology is equivalent to the subspace topology. Thus continuity of $\mathfrak{r}_{k}$ follows from continuity of the map
\[
\mathcal{L}(\Complex^{k}\oplus \Complex^{d_{0}-k},V_{0})\times\inj(\Complex^{k},V_{1})\times s(\Complex^{d_{0}-k})\to \svo\times\aich
\]
\[
((\zeta,\eta),\gamma_{0},\psi_{0})\mapsto ((\log(e_{0}(\rho(\gamma_{0}\circ\zeta^{\dag})))-\exp(-\eta\circ\psi_{0}\circ\eta^{\dag}))\circ(1_{V_{0}}-\zeta\zeta^{\dag})\oplus \log(\rho(\gamma_{0}\circ\zeta^{\dag}))\circ\zeta\zeta^{\dag},\gamma_{0}\circ\zeta^{\dag})
\]
which is continuous as standard.
\end{proof}

\begin{proposition}\label{sigmainjchomeoextended} Define
\[
\mathcal{J}_{k}':=\{(W,\delta,\psi):W\in
G_{k}(V_{0}),\delta\in\inj(W,V_{1})^{c},\psi\in s(W^{\bot})\}.
\]
Topologize $\mathcal{J}_{k}'$ as a subspace of $\tilde{Z}_{k}$ from \ref{Zk} and define $\mathcal{J}_{k}:=(\mathcal{J}_{k}')_{\infty}$; note this space has a $G$--action inherited from $\tilde{Z}_{k}$. Recall $\rho$, $\sigma$ and $\lambda_{k-1}$ from \ref{themaprho} and \ref{themapsigma} and \ref{themaplambda}, the maps $\mathfrak{f}_{k}\co\tilde{X}_{k-1}'\backslash Y_{k}'\to \Real\times \mathcal{J}_{k}'$ and $\mathfrak{g}_{k}\co\Real\times \mathcal{J}_{k}'\to \tilde{X}_{k-1}'\backslash Y_{k}'$ given by
\[
\mathfrak{f}_{k}\co(\alpha,\theta)\mapsto\left(e_{d_{0}-k}(\alpha),P_{k}(\alpha),-\theta\circ \lambda_{k-1}(\alpha)|_{P_{k}(\alpha)},-\log((e_{d_{0}-k}(\alpha)-\alpha)|_{P_{k}(\alpha)^{\bot}})\right)
\]
\[
\left((t-\Exp(-\psi))|_{W^{\bot}}\oplus (\rho(\delta)+t)|_{W}, -\sigma(\delta)\right)\mapsfrom(t,W,\delta,\psi)\oc\mathfrak{g}_{k}
\]
are well-defined continuous $G$--maps that are inverses of each other, hence $\tilde{X}_{k-1}'\backslash Y_{k}'\cong \Real\times \mathcal{J}_{k}'$ and thus $\tilde{X}_{k-1}/ Y_{k}\cong\Sigma \mathcal{J}_{k}$.
\end{proposition}
\begin{proof} As above, the only issues with this proof are the continuity statements. To show that $\mathfrak{f}_{k}$ is continuous, equip $\tilde{X}_{k-1}'\backslash Y_{k}$ with the quotient topology of \ref{Xkasspaces} and topologize $\Real\times\tilde{J}_{k}'$ as a subspace of $\Real\times\tilde{Z}_{k}'$ and hence as a subspace of $\Real\times G_{k}(V_{0})\times\aich\times\svo$ via \ref{topologizingthefibrebundles}. The continuity of $\mathfrak{f}_{k}$ then follows from continuity of the map
\[
\{(\alpha,\theta)\in\svo\times\Ell:\Dim(P_{k}(\alpha))=k\}\to \Real\times G_{k}(V_{0})\times\aich\times\svo
\]
\[
(\alpha,\theta)\mapsto(e_{d_{0}-k}(\alpha),P_{k}(\alpha),-\theta\circ(\alpha-e_{d_{0}-k}(\alpha)),-\log(e_{d_{0}-k}(\alpha)-\alpha)).
\]
Continuity of this map follows from \ref{Pkiscont}. As in the proof of \ref{thetaEalphaextended} $\mathcal{J}_{k}'$ can also be an equipped with an equivalent quotient topology given by the map
\begin{align*}
\mathcal{L}(\Complex^{k}\oplus \Complex^{d_{0}-k},V_{0})\times\inj(\Complex^{k},V_{1})^{c}\times s(\Complex^{d_{0}-k})&\to \mathcal{J}_{k}'\\
((\zeta,\eta),\gamma_{0},\psi_{0})&\mapsto (\IM(\zeta),\gamma_{0}\circ \zeta^{\dag},\eta\circ\psi_{0}\circ\eta^{\dag}).
\end{align*}
Equip $\tilde{X}_{k-1}'\backslash Y_{k}$ with the subspace topology of \ref{Xkothertopology}, then continuity of $\mathfrak{g}_{k}$ follows from continuity of the map
\[
\Real\times\mathcal{L}(\Complex^{k}\oplus \Complex^{d_{0}-k},V_{0})\times\inj(\Complex^{k},V_{1})^{c}\times s(\Complex^{d_{0}-k})\to \svo\times\aich
\]
\[
((\zeta,\eta),\gamma_{0},\psi_{0})\mapsto((t-\exp(-\eta\circ\psi_{0}\circ\eta^{\dag}))\circ(1_{V_{0}}-\zeta\zeta^{\dag})\oplus(\rho(\gamma_{0}\circ\zeta^{\dag})+t)\circ\zeta\zeta^{\dag},\gamma_{0}\circ\zeta^{\dag})
\]
which is continuous as standard.
\end{proof}

There is a unique map $\chi'$ making
\[
\xymatrix{\frac{\tilde{X}_{k}}{Y_{k}}\ar[r]_{\cong}^{\mathfrak{q}_{k}}\ar[d]_{\varpi_{k}}&\mathcal{I}_{k}\ar[d]^{\chi'}\\
\frac{\tilde{X}_{k-1}}{Y_{k}}\ar[r]^{\cong}_{\mathfrak{f}_{k}}&\Sigma\mathcal{J}_{k}}
\]
strictly commute. We have a sequence
\[
\mathcal{I}_{k}\overset{\chi_{k}}{\to}\Sigma \mathcal{J}_{k}\overset{-\Sigma i_{k}}{\to} G_{k}(V_{0})^{\Real\oplus\Hom(T,V_{1})\oplus s(T^{\bot})}\overset{-\Sigma p_{k}}{\to} \Sigma \mathcal{I}_{k}
\]
with $p_{k}$ induced from the fibrewise collapses $S^{\Hom(W,V_{1})}\to \inj(W,V_{1})_{\infty}$. The map $i_{k}$ is induced from the fibrewise inclusions $\inj(W,V_{1})^{c}_{\infty}\to S^{\Hom(W,V_{1})}$ and $\chi_{k}$ is induced from the fibrewise maps $\chi_{W}=\mathfrak{B}_{g}\co\inj(W,V_{1})_{\infty}\to \Sigma\inj(W,V_{1})^{c}_{\infty}$; here $g$ is the map
\begin{align*}
g\co\frac{D_{+}(k)}{D_{0}(k)}&\to \Sigma D_{0}(k)\\
(t_{0},\ldots,t_{k-1})&\mapsto (\log(t_{0}),0,\log(t_{1})-\log(t_{0}),\ldots,\log(t_{k-1})-log(t_{0})).
\end{align*}
It follows from \ref{cofibresequenceandbundles} that this is a cofibre sequence. It is a simple task to check that $\chi'=\chi_{k}$ and hence we have a cofibre sequence
\[
\frac{\tilde{X}_{k}}{Y_{k}}\overset{\varpi}{\to}\frac{\tilde{X}_{k-1}}{Y_{k}}\to G_{k}(V_{0})^{\Real\oplus \Hom(T,V_{1})\oplus s(T^{\bot})}\to \Sigma \frac{\tilde{X}_{k}}{Y_{k}}.
\]
Thus by applying \ref{quotientofacofibresequence}, applying $\Sigma^{\infty}$ and smashing by $S^{-\svo}$ we observe that a cofibre sequence
\[
X_{k}\overset{\pi_{k}}{\to}X_{k-1}\to G_{k}(V_{0})^{\Real\oplus\Hom(T,V_{1}-V_{0})\oplus s(T)}\to \Sigma X_{k}
\]
exists. The above work is easily checked to interact well with the stated group actions. Theorem~\ref{TheMainResult} then follows.
 
\subsection{Explicit maps in the sequences}\label{Explicit maps in the sequence}

While the above work proves Theorem~\ref{TheMainResult} it is somewhat unsatisfactory as it only theoretically demonstrates that there is a cofibre sequence. We now state maps forming a sequence
\[
\tilde{X}_{k}\overset{\tilde{\pi}_{k}}{\to}\tilde{X}_{k-1}\overset{\tilde{\delta}_{k}}{\to} G_{k}(V_{0})^{\Real\oplus\Hom(T,V_{1})\oplus s(T^{\bot})}\overset{-\Sigma\tilde{\phi}_{k}}{\to} \Sigma \tilde{X}_{k}.
\]
This will be the unstable sequence building the sequence in Theorem~\ref{TheMainResult}. We only need to do this when $k<d_{0}$---at the top of the tower we didn't take a quotient so already have explicit unstable maps.

\begin{definition}\label{themapdeltak} Define
\[
\tilde{\delta}_{k}\co\tilde{X}_{k-1}\overset{\operatorname{coll}}{\to}\frac{\tilde{X}_{k-1}}{Y_{k}}\overset{(\mathfrak{f}_{k})_{\infty}}{\cong}\Sigma\mathcal{J}_{k}\overset{-\Sigma i_{k}}{\to} G_{k}(V_{0})^{\Real\oplus \Hom(T,V_{1})\oplus s(T^{\bot})}
\]
and set $\delta_{k}:=\Sigma^{-\svo}\Sigma^{\infty}\tilde{\delta}_{k}$. Here $\operatorname{coll}$ is the standard collapse, $\mathfrak{f}_{k}$ was defined in \ref{sigmainjchomeoextended} and $i_{k}$ is the fibrewise inclusion.
\end{definition}

\begin{definition}\label{themapphik} Define
\begin{align*}
\tilde{\phi}_{k}\co G_{k}(V_{0})^{\Hom(T,V_{1})\oplus s(T^{\bot})}&\to \tilde{X}_{k}\\
(W,\gamma,\psi)&\mapsto (\psi|_{W^{\bot}}\oplus (\rho(\gamma)+e_{top}(\psi))|_{W},-\sigma(\gamma))
\end{align*}
and set $\phi_{k}:=\Sigma^{-\svo}\Sigma^{\infty}\tilde{\phi}_{k}$. Here $e_{top}(\psi)$ is the top eigenvalue of $\psi$ under the standard ordering and we recall $\rho$ and $\sigma$ from \ref{themaprho} and \ref{themapsigma}.
\end{definition}

We have already covered why $\tilde{\delta}_{k}$ is well-defined and continuous, it is a simple exercise to check that $\tilde{\phi}_{k}$ is also well-defined (i.e.\ the unbased map is proper) and continuous using similar techniques to those used in \ref{thetaEalphaextended} and \ref{sigmainjchomeoextended}. Recall $p_{k}\co G_{k}(V_{0})^{\Hom(T,V_{1})\oplus s(T^{\bot})}\to \mathcal{I}_{k}$ to be the fibrewise collapse, set $c_{k}\co \tilde{X}_{k}\to \tilde{X}_{k}/Y_{k}$ to be the collapse and recall $\mathfrak{r}_{k}$ from \ref{thetaEalphaextended}. We have a homotopy commutative diagram
\[
\xymatrix{Y_{k}\ar@{>->}[r]\ar[d]_{1}&\tilde{X}_{k}\ar[d]_{\tilde{\pi}_{k}}\ar@{->>}[r]_{c_{k}}&\frac{\tilde{X}_{k}}{Y_{k}}\ar[d]&G_{k}(V_{0})^{\Hom(T,V_{1})\oplus s(T^{\bot})}\ar@/_1.5pc/@{.>}[ll]_{\tilde{\phi}_{k}?}\ar[l]_{(\mathfrak{r}_{k})_{\infty}\circ p_{k}\phantom{xxxxxxxx}}\\
Y_{k}\ar@{ >->}[r]&\tilde{X}_{k-1}\ar@/_2.5pc/[rru]|\bigcirc_{\phantom{xx}\tilde{\delta_{k}}}\ar@{->>}[r]&\frac{\tilde{X}_{k-1}}{Y_{k}}\ar[ur]|\bigcirc}
\]
where the inner triangle is our cofibre sequence. If we can show adding $\tilde{\phi}_{k}$ to the diagram maintains its homotopy-commutativity then we conclude that
\[
\tilde{X}_{k}\overset{\tilde{\pi}_{k}}{\to}\tilde{X}_{k-1}\overset{\tilde{\delta}_{k}}{\to} G_{k}(V_{0})^{\Real\oplus\Hom(T,V_{1})\oplus s(T^{\bot})}\overset{-\Sigma\tilde{\phi}_{k}}{\to} \Sigma \tilde{X}_{k}
\]
is a cofibre sequence. To proceed we extend our functional calculus theory from Section~\ref{Extended functional calculus}. We have a homeomorphism
\begin{align*}
f\co D(d_{0}-k)\wedge D_{+}(k)&\to D(d_{0})\\
(s,t)&\mapsto (s,s_{top}+t).
\end{align*}

\begin{definition}\label{ExtendedFacialDefinition} We say a map $g\co D(d_{0}-k)\wedge D_{+}(k)\to D(d_{0})$ is facial if the composite $g\circ f^{-1}\co D(d_{0})\to D(d_{0})$ is. In particular $f$ is facial.
\end{definition}

The next two results are easy to check.

\begin{lemma}\label{themapp} Let $i\co\Complex^{k}\to \Complex ^{d_{0}}$ be a choice of inclusion sending $\Complex^{k}$ to the last $k$ copies of $\Complex$ in $\Complex^{d_{0}}$. Recall $\tilde{Z}_{k}$ from \ref{Zk} and $\Delta$ from \ref{DiagonalMatrixMap} and define
\begin{align*}
p'\co\mathcal{L}(\Complex^{d_{0}},V_{0})\times\mathcal{L}(\Complex^{k},V_{1})\times  D'(d_{0}-k)\times D'_{+}(k)&\to \tilde{Z}_{k}\\
(\lambda,\mu,s,t)&\mapsto (\lambda(i(\Complex^{k})),-\mu\circ \Delta(t)\lambda^{-1}|_{\lambda(i(\Complex^{k}))},\lambda\Delta(s)\lambda^{-1}|_{\lambda(i(\Complex^{k}))^{\bot}}).
\end{align*}
Then $p'$ is a continuous proper surjection and hence $p:=(p')_{\infty}$ is a quotient map.
\end{lemma}

\begin{lemma}\label{themapq} Define
\begin{align*}
q'\co \mathcal{L}(\Complex^{d_{0}},V_{0})\times\mathcal{L}(\Complex^{k},V_{1})\times D'(d_{0})&\to \tilde{X}'_{k}\\
(\lambda,\mu,t')&\mapsto (\lambda\Delta(t')\lambda^{-1},-\mu\circ \lambda^{-1}|_{P_{k}(\lambda\Delta(t')\lambda^{-1})}).
\end{align*}
Then $q'$ is a continuous proper surjection and hence $q:=(q')_{\infty}$ is a quotient map.
\end{lemma}

\begin{proposition}\label{fnalvarC} Let $g\co D(d_{0}-k)\wedge D_{+}(k)\to D(d_{0})$ be facial. Then there exists a unique map 
\[
\mathfrak{C}_{g}\co G_{k}(V_{0})^{\Hom(T,V_{1})\oplus s(T^{\bot})}\to \tilde{X}_{k}
\]
making
\[
\xymatrix{\mathcal{L}(\Complex^{d_{0}},V_{0})_{\infty}\wedge\mathcal{L}(\Complex^{k},V_{1})_{\infty}\wedge D(d_{0}-k)\wedge D_{+}(k)\ar[d]_{1\wedge1\wedge g}\ar[r]^{\phantom{xxxxxxxxxxxx}p}&G_{k}(V_{0})^{\Hom(T,V_{1})\oplus s(T^{\bot})}\ar[d]^{\mathfrak{C}_{g}}\\
\mathcal{L}(\Complex^{d_{0}},V_{0})_{\infty}\wedge\mathcal{L}(\Complex^{k},V_{1})_{\infty}\wedge D(d_{0})\ar[r]_{\phantom{xxxxxxxxxxxx}q}&\tilde{X}_{k}}
\]
commute; moreover, the associated map
\begin{align*}
\mathfrak{C}\co\FMap(D(d_{0}-k)\wedge D_{+}(k),D(d_{0}))&\to \Map(G_{k}(V_{0})^{\Hom(T,V_{1})\oplus s(T^{\bot})}, \tilde{X}_{k})\\
g&\mapsto \mathfrak{C}_{g}
\end{align*}
is continuous. Furthermore we have an explicit description of $\mathfrak{C}_{g}(W,\gamma,\psi)$. Choose an orthonormal basis for $W^{\bot}$ of eigenvectors $v_{0},\ldots,v_{d_{0}-k-1}$ of $\psi$ with eigenvalues $e_{0}\leqslant\ldots\leqslant e_{d_{0}-k-1}$ and choose and orthonormal basis for $W$ of eigenvectors $v_{d_{0}-k},\ldots,v_{d_{0}-1}$ of $\gamma^{\dag}\gamma$ with eigenvalues
$e_{d_{0}-k}^{2}\leqslant\ldots\leqslant e_{d_{0}-1}^{2}$. If $g(e)=s$, then $\mathfrak{C}_{g}(W,\gamma,\psi)$ maps to $(\alpha,\theta)$ where $\alpha$ is the self-adjoint transformation of $V_{0}$ with eigenvectors $v_{i}$ and eigenvalues $s_{i}$ and $\theta= -\sigma(\gamma)$, recalling $\sigma$ from \ref{themapsigma}.
\end{proposition}
\begin{proof} Let $p(\lambda,\mu,s,t)=p(\lambda',\mu',s',t')=(W,\gamma,\psi)$, we claim that $q(\lambda,\mu,g(s,t))=q(\lambda',\mu',g(s',t'))$. Set $\zeta:=\lambda^{-1}\lambda'\co\Complex^{d_{0}}\to \Complex^{d_{0}}$. Then as $\lambda(i(\Complex^{k}))=\lambda'(i(\Complex^{k}))$ we have $\zeta(i(\Complex^{k}))=i(\Complex^{k})$ and thus $\zeta$ splits into $\zeta_{0}\co\Complex^{k}\to \Complex^{k}$ and $\zeta_{1}\co\Complex^{d_{0}-k}\to \Complex^{d_{0}-k}$.  Suppressing the notation of $i$ we have $\rho(\gamma)=\lambda\Delta(t)\lambda^{-1}|_{W}=\lambda'\Delta(t')\lambda'^{-1}|_{W}$. This implies
$\zeta_{0}^{-1}\Delta(t)\zeta_{0}=\Delta(t')$, hence $t=t'$ and $\zeta_{0}$ is in the centralizer of $\Delta(t)$. Similarly $s=s'$ and $\zeta_{1}$ is in the centralizer of $\Delta(s)$. Now, let $\Delta(s\oplus t)$ denote the diagonal matrix with the entries $s$ and then $t$. It is clear that $\zeta$ is in the centralizer of $\Delta(s\oplus t)$ and hence in the centralizer of $\Delta(g(s,t))$.

We want to show that $-\mu\circ \lambda^{-1}|_{P_{k}(\lambda\Delta(g(s,t))\lambda^{-1})}=-\mu'\circ \lambda'^{-1}|_{P_{k}(\lambda'\Delta(g(s',t'))\lambda'^{-1})}$ so consider $\mu^{-1}\mu'$ restricted to the top $n$ copies of $\Complex$ in $\Complex^{k}$; $n$ is the maximal value such that $\Delta(t)|_{\Complex^{n}}$ is invertible. On this restriction it is clear that $\mu^{-1}\mu'$ agrees with $\zeta_{0}$. This fact, combined with the above paragraph, is enough to show that $q(\lambda,\mu,g(s,t))=q(\lambda',\mu',g(s',t'))$. Further our described map clearly makes the diagram commute. As $p$ is surjective $\mathfrak{C}_{g}$ is unique. Similar to the proofs of \ref{fnalvarA} and \ref{fnalvarB} the map $\mathfrak{C}_{g}$ is continuous as $p$ is a quotient.

To show $\mathfrak{C}$ is continuous we show that the adjoint $\mathfrak{C}^{\#}$ is continuous. Use the shorthand $\mathcal{L}:=\mathcal{L}(\Complex^{d_{0}},V_{0})_{\infty}\wedge\mathcal{L}(\Complex^{k},V_{1})_{\infty}$. Let $\text{eval}$ be the map
\begin{align*}
\text{eval}\co\FMap (D(d_{0}-k)\wedge D_{+}(k),D(d_{0}))\wedge\mathcal{L}\wedge D(d_{0}-k)\wedge D_{+}(k)&\to\mathcal{L}\wedge D(d_{0})\\
(g,\lambda,\mu,s,t)&\mapsto(\lambda,\mu,g(s,t)).
\end{align*}
We have a commutative diagram
\[
\xymatrix{\FMap (D(d_{0}-k)\wedge D_{+}(k),D(d_{0}))\wedge\mathcal{L}\wedge D(d_{0}-k)\wedge D_{+}(k)\ar[r]^{\phantom{xxxxxxxxxxxxxxxxxxxxxx}\text{eval}}\ar[d]_{1\wedge p}&\mathcal{L}\wedge D(d_{0})\ar[d]^{q}\\
\FMap (D(d_{0}-k)\wedge
D_{+}(k),D(d_{0}))\wedge G_{k}(V_{0})^{\Hom(T,V_{1})\oplus s(T^{\bot})}\ar[r]_{\phantom{xxxxxxxxxxxxxxxxxxxxxx}\mathfrak{C}^{\#}}&\tilde{X}_{k}}
\]
and hence $\mathfrak{C}^{\#}\circ (1\wedge p)=q\circ\text{eval}$ is continuous. The map $(1\wedge p)$ is a quotient, thus $\mathfrak{C}^{\#}$ and $\mathfrak{C}$ are continuous.
\end{proof}

We can immediately note the following fact.

\begin{lemma}\label{phikisfnalcalc} $\tilde{\phi}_{k}=\mathfrak{C}_{f}$ for the particular map $f$ defined immediately before \ref{ExtendedFacialDefinition}.
\end{lemma}

We can extend the functional calculus to the following result.

\begin{corollary}\label{fnalvarD}  Recalling $\mathcal{I}_{k}'$ from \ref{thetaEalphaextended}, the maps $p'$ and $q'$ restrict to
\[
p'\co\mathcal{L}(\Complex^{d_{0}},V_{0})\times\mathcal{L}(\Complex^{k},V_{1})\times D'(d_{0}-k)\times D'_{+}(k)\backslash D_{0}'(k)\to \mathcal{I}_{k}'
\]
\[
q'\co\mathcal{L}(\Complex^{d_{0}},V_{0})\times\mathcal{L}(\Complex^{k},V_{1})\times (D'(d_{0})\backslash F_{d_{0}-k-1}(D'(d_{0})))\to \tilde{X}'_{k}\backslash Y'_{k}.
\]
Let $g'\co D(d_{0}-k)\wedge D_{+}(k)/D_{0}(k)\to D(d_{0})/F_{d_{0}-k-1}(D(d_{0}))$ be facial.
There exists a unique map $\mathfrak{D}_{g'}$ holding the properties stated in \ref{fnalvarC} and making
\[
\xymatrix{\mathcal{L}(\Complex^{d_{0}},V_{0})_{\infty}\wedge\mathcal{L}(\Complex^{k},V_{1})_{\infty}\wedge D(d_{0}-k)\wedge\frac{D_{+}(k)}{D_{0}(k)} \ar[d]_{1\wedge1\wedge g'}\ar[r]^{\phantom{xxxxxxxxxxxxxxxxxxxx}p}&\mathcal{I}_{k}\ar[d]^{\mathfrak{D}_{g'}}\\
\mathcal{L}(\Complex^{d_{0}},V_{0})_{\infty}\wedge\mathcal{L}(\Complex^{k},V_{1})_{\infty}\wedge
\frac{D(d_{0})}{F_{d_{0}-k-1}(D(d_{0}))}\ar[r]_{\phantom{xxxxxxxxxxxxxxxxxxxx}q}&\frac{\tilde{X}_{k}}{Y_{k}}}
\]
commute.
\end{corollary}

The next three results are simple to prove.

\begin{lemma}\label{collapseinfnalcalc} Let $h\co D(d_{0})\to D(d_{0})/F_{d_{0}-k-1}(D(d_{0}))$ be the collapse. Then
\[
\xymatrix{\mathcal{L}(\Complex^{d_{0}},V_{0})_{\infty}\wedge\mathcal{L}(\Complex^{k},V_{1})_{\infty}\wedge D(d_{0})\ar[r]^{\phantom{xxxxxxxxxxxx}q}\ar[d]_{1\wedge h}&\tilde{X}_{k}\ar[d]^{c_{k}}\\
\mathcal{L}(\Complex^{d_{0}},V_{0})_{\infty}\wedge\mathcal{L}(\Complex^{k},V_{1})_{\infty}\wedge \frac{D(d_{0})}{F_{d_{0}-k-1}(D(d_{0}))}\ar[r]_{\phantom{xxxxxxxxxxxxxxxxx}q}&\frac{\tilde{X}_{k}}{Y_{k}}
}
\]
commutes.
\end{lemma}

\begin{lemma}\label{theothercollapseinfnalcalc} Let $h'\co D(d_{0}-k)\wedge D_{+}(k)\to D(d_{0}-k) \wedge (D_{+}(k)/D_{0}(k))$ be the collapse. Then 
\[
\xymatrix{\mathcal{L}(\Complex^{d_{0}},V_{0})_{\infty}\wedge\mathcal{L}(\Complex^{k},V_{1})_{\infty}\wedge D(d_{0}-k)\wedge D_{+}(k)
\ar[d]_{1\wedge 1\wedge h'}\ar[r]^{\phantom{xxxxxxxxxxxx}p}&G_{k}(V_{0})^{\Hom(T,V_{1})\oplus s(T^{\bot})}\ar[d]^{p_{k}}\\
\mathcal{L}(\Complex^{d_{0}},V_{0})_{\infty}\wedge\mathcal{L}(\Complex^{k},V_{1})_{\infty}\wedge
D(d_{0}-k)\wedge
\frac{D_{+}(k)}{D_{0}(k)}\ar[r]_{\phantom{xxxxxxxxxxxx}p}&\mathcal{I}_{k}}
\]
commutes.
\end{lemma}

\begin{lemma}\label{comparingCandD} Let $g$ and $g'$ be such that $h\circ g=g'\circ h'$. Then $\mathfrak{D}_{g'}\circ p_{k}=c_{k}\circ \mathfrak{C}_{g}$.
\end{lemma}

We can now work in the functional calculus. We recall \ref{phikisfnalcalc} and state the below lemma, again the proof is standard.

\begin{lemma}\label{Rkisfnalcalc} Define
\begin{align*}
g'\co D(d_{0}-k)\wedge D_{+}(k)/D_{0}(k)&\to D(d_{0})/F_{d_{0}-k-1}(D(d_{0}))\\
(s,t)&\mapsto (\log(t_{0})-\Exp(-s),\log(t)).
\end{align*}
Then $\mathfrak{D}_{g'}= (\mathfrak{r}_{k})_{\infty}$.
\end{lemma}

We have a diagram
\[
\xymatrix{&D(d_{0})\ar[d]^{h}\\
D(d_{0}-k)\wedge D_{+}(k)\ar[r]_{g'\circ h'}\ar[ur]^{f}&
\frac{D(d_{0})}{F_{d_{0}-k-1}(D(d_{0}))}}
\]
which we claim commutes up to facial homotopy. There is a map $\bar{f}\co D(d_{0}-k)\wedge D_{+}(k)/D_{0}(k)\to D(d_{0})/F_{d_{0}-k-1}(D(d_{0}))$ making
\[
\xymatrix{D(d_{0}-k)\wedge D_{+}(k)\ar[d]_{h'}\ar[r]^{f}&D(d_{0})\ar[d]^{h}\\
D(d_{0}-k)\wedge
\frac{D_{+}(k)}{D_{0}(k)}\ar[r]_{\bar{f}}&\frac{D(d_{0})}{F_{d_{0}-k-1}(D(d_{0}))}}
\]
strictly commute. If $\bar{f}$ is homotopic to $g'$ through a facial homotopy then it follows that the diagram is homotopy commutative.

\begin{proposition}\label{homotopycommutativityinthiscase} $\bar{f}\simeq g'$ via a facial homotopy.
\end{proposition}
\begin{proof} First note that $\bar{f}$ is a facial homeomorphism. Also note that $\exp\co D(k)\to D_{+}(k)/D_{0}(k)$ is a facial homeomorphism with inverse $\log$. Let $m\in\FMap (D(d_{0}-k)\wedge D_{+}(k)/D_{0}(k),D(d_{0})/F_{d_{0}-k-1}(D(d_{0})))$. We have an associated composition
\[
m'=(1\wedge\log)\circ\bar{f}^{-1}\circ m\circ (1\wedge\exp)\co D(d_{0}-k)\wedge D(k)\to D(d_{0}-k)\wedge D(k)
\]
which is facial as all components in the composition are facial. Hence we have induced maps $\bar{f}'$, $g''\co D(d_{0}-k)\wedge D(k)\to D(d_{0}-k)\wedge D(k)$ which we claim are homotopic via a facial homotopy. The homotopy type of $\FMap(D(d_{0}-k)\wedge D(k))$ is known; a facial map $m'\co D(d_{0}-k)\wedge D(k)\to D(d_{0}-k)\wedge D(k)$ must by necessity be of the form $m_{0}\wedge m_{1}$ for $m_{0}\co D(d_{0}-k)\to D(d_{0}-k)$ and $m_{1}\co D(k)\to D(k)$ facial. Hence we have a copy of $S^{2}$ embedded in $D(d_{0}-k)\wedge D(k)$ arising from the copies of $S^{1}$ embedded in $D(d_{0}-k)$ and $D(k)$ and thus an induced map $m''\co S^{2}\to S^{2}$ which via degree governs the facial homotopy type of $m'$---this is an immediate corollary of Theorem~\ref{facialhomotopiesviadegrees}. The maps $\bar{f}''$, $g'''\co S^{2}\to S^{2}$ built from $\bar{f}'$ and $g''$ are
\[
\bar{f}''\co(s,t)\mapsto (s,t)
\]
\[
g'''\co(s,t)\mapsto (t-e^{-s},-s).
\]
Demonstrating that $g'''$ is degree $1$ will complete the proof. Consider the unbased map $\Real\times\Real\to \Real\times\Real$ given by $(s,t)\mapsto (t-e^{-s},-s)$. This map has derivative matrix
\[
\left(\begin{array}{cc}e^{-s}& 1\\
-1 & 0\end{array}\right)
\]
which has determinant $1$. It follows that $g'''$ is degree $1$ and hence $\bar{f}\simeq g'$.
\end{proof}

Thus the diagram
\[
\xymatrix{&D(d_{0})\ar[d]^{h}\\
D(d_{0}-k)\wedge D_{+}(k)\ar[r]_{g'\circ h'}\ar[ur]^{f}&
\frac{D(d_{0})}{F_{d_{0}-k-1}(D(d_{0}))}}
\]
commutes up to homotopy. This result can be pulled up through the functional calculus via \ref{comparingCandD}, and further it is easy to check that equivariance is satisfied throughout. This proves the following proposition.

\begin{proposition}\label{wehavealift} $(\mathfrak{r}_{k})_{\infty}\circ p_{k}\simeq c_{k}\circ\tilde{\phi}_{k}$ and hence we have a cofibre sequence
\[
\xymatrix@C=1.5cm{\tilde{X}_{k}\ar[d]_{\tilde{\pi}_{k}}&G_{k}(V_{0})^{\Hom(T,V_{1})\oplus s(T^{\bot})}\ar[l]_{\tilde{\phi}_{k}\phantom{xxxxxxx}}\\
\tilde{X}_{k-1}\ar[ru]|\bigcirc_{\phantom{xxx}\tilde{\delta_{k}}}&}
\]
for the stated maps $\tilde{\pi}_{k}$, $\tilde{\delta}_{k}$ and $\tilde{\phi}_{k}$.
\end{proposition}

It follows that
\[
\xymatrix@C=1.5cm
{X_{k}\ar[d]_{\pi_{k}}&G_{k}(V_{0})^{\Hom(T,V_{1} - V_{0})\oplus s(T)}\ar[l]_{\phi_{k}\phantom{xxxxxxx}}\\
X_{k-1}\ar[ur]|\bigcirc_{\phantom{xxx}\delta_{k}}&}
\]
is a cofibre sequence.

\section{Gysin maps and residues}\label{Gysin maps and residues}

We take a brief detour from studying the tower in order to establish a result linking certain Gysin maps with residue theory. This will then be used to produce an obstruction, previously stated as Theorem\ref{IntroObstruction}, to a splitting of Theorem~\ref{TheMainResult} in many cases. 

Our framework is as follows, let $G$ be a compact Lie group, let $V$ be a complex $G$--representation of dimension $d$ and let $j\co PV\rightarrowtail W$ be an equivariant embedding of $PV$ into a representation $W$. There is an associated Pontryagin--Thom collapse map $j^{!}\co S^{W}\to PV^{W\ominus\tau_{PV}}$ where $\tau_{PV}$ is the tangent bundle over $PV$. Hence there is a stable collapse
\[
\Sigma^{-W}j^{!}\co S^{0}\to PV_{0}^{-\tau_{PV}}.
\]
Let
\[
j_{!}=(\Sigma^{-W}j^{!})^{*}\co\redKstar(PV^{-\tau})\to \Kstar(S^{0}).
\]
be the associated Gysin map in equivariant $K$--theory, we claim that we can describe $j_{!}$ as an algebraic geometry style residue map. The non-equivariant version of this result was first proved by Quillen in \cite{QuillenCobordismFGL} using formal methods while for $G$ a finite abelian group the result was proved by Strickland in \cite[Theorem $21.35$]{StricklandMulticurves}, again using formal methods. We provide a purely geometric proof for $G$ any compact Lie group. We note here that the formal proof of \cite[Theorem $21.35$]{StricklandMulticurves} does actually pass to the general $G$ case in $K$--theory but our presented method bypasses many technicalities. Much of the detail and hard work of this proof lies in the work of Strickland, \cite[Section $21$]{StricklandMulticurves}.

First note that all $\Kone$--groups that may occur are zero, hence we restrict discussion to $\Kzero$. We recall the workings of a residue map. Consider an expression $f(x)\,dx$ where $f(x)=p(x)/q(x)$ for $p$ a polynomial and $q$ a monic polynomial. Then $f(x)$ can be expanded to a Laurent series with coefficients $b_{i}$ and the residue map is $\operatorname{res}(f(x)\,dx)= b_{-1}$, see \cite[Definition $21.26$]{StricklandMulticurves} for more detail. We note that $\Kzero(S^{0})$ is the complex representation ring $R(G)$ of $G$ and define
\[
f_{V}(z):= \sum_{k=0}^{d}z^{d-k}(-1)^{k}.\lambda^{k}(V)\in R(G)[z],
\]
a polynomial with coefficients constructed from the exterior powers $\lambda^{k}(V)$ of $V$. We also take the convention that $f_{V}(z)=\sum_{i=0}^{d} a_{i} z^{i}$ with $a_{d}=1$ and $a_{0}$ invertible. It is standard that the equivariant $K$--theory of $PV$ is
\[
\Kzero(PV)\cong \frac{R(G)[z]}{f_{V}(z)}\cong R(G)\{z^{i}:0\leqslant i<d\}
\]
and moreover that if $u_{-\tau_{PV}}$ is the Thom class of $-\tau_{PV}$ then $\redKzero(PV^{-\tau_{PV}})\cong \Kzero (PV).u_{-\tau_{PV}}$ by the Thom Isomorphism Theorem.

\begin{proposition}\label{GysinMapsAndResidue} We can identify the Thom class $u_{-\tau_{PV}}$ with $dz/f_{V}(z)$ and the map
\begin{align*}
j_{!}\co\frac{R(G)[z]}{f_{V}(z)}.u_{-\tau_{PV}}&\to R(G)\\
g(z).u_{-\tau_{PV}}&\mapsto j_{!}(g(z).u_{-\tau_{PV}})
\end{align*}
with the residue map
\[
\frac{g(z)}{f_{V}(z)}\,dz\mapsto\operatorname{res}\left(\frac{g(z)}{f_{V}(z)}\,dz\right). 
\] 
\end{proposition}
\begin{proof} We actually prove an equivalent problem. Let $\pi\co PV\to\text{pt}$ be the projection. We show that the related stable Gysin map
\begin{align*}
\pi_{!}\co\redKzero(PV^{-\tau_{PV}})\cong R(G)\{z^{i}:0\leqslant i<d\}.u_{-\tau_{PV}}&\to \Kzero (S^{0})\cong R(G)\\
z^{i}.u_{-\tau_{PV}}&\mapsto r_{i}
\end{align*}
is a residue map by determining each $r_{i}$. Consider the diagonal $\delta\co PV\to PV\times PV$. We note that $(1\times\pi)\circ\delta$ is the identity. There is an associated Gysin map
\[
\delta_{!}\co \redKzero(PV^{\tau_{PV}})\cong\redKzero(PV).u_{\tau_{PV}}\to \Kzero(PV\times PV)\cong R(G)\{z^{i}w^{j}:0\leqslant i,j<d\}
\]
which sends $u_{\tau_{PV}}$ to some element $e$ say. The geometry of the stable collapse map $\delta^{!}\co PV\times PV\to PV^{\tau_{PV}}$ is known, it is a quotient as detailed in \cite[Corollary $21.37$]{StricklandMulticurves}. Thus $e$ is annihilated by $z-w$, i.e.\ $\delta_{!}(z^{i}u_{\tau_{PV}})=z^{i}e=w^{i}e$. We observe (cf.\ \cite[Example $21.6$]{StricklandMulticurves}) that
\[
e=\frac{f_{V}(z)-f_{V}(w)}{z-w}=\sum_{0\leqslant i+j<d} a_{i+j+1}z^{i}w^{j}.
\]
Now tensor throughout the map $\delta_{!}$ with $\redKzero(PV_{\infty}\wedge PV^{-\tau_{PV}})$ over $\Kzero(PV\times PV)$ to get a map
\[
\lambda\co \Kzero(PV)\to \redKzero(PV_{\infty}\wedge PV^{-\tau_{PV}})
\]
and observe from the geometric relationship of $\pi$ and $\delta$ that $(1\otimes \pi_{!})\circ\lambda\co \Kzero(PV)\to \Kzero(PV)$ is the identity. By the description of $\delta_{!}$ it follows that $\lambda(1)=e.(1\otimes u_{-\tau_{PV}})$ and hence as $(1\otimes \pi_{!})\circ\lambda=1$ that
\[
1=\sum_{0\leqslant i+j<d} a_{i+j+1} z^{i} r_{j}.
\]
This equality is satisfied by $r_{d-1}=1$ and $r_{j}=0$ for $j<d-1$ and thus
\[
\pi_{!}(z^{i}u_{-\tau_{PV}})=\left\{\begin{array}{ll}1&i=d-1\\
0&i<d-1.\end{array}\right.
\]
By \cite[Lemma $21.28$]{StricklandMulticurves} we can make the stated identifications and recognize this map as the residue map.
\end{proof}

\subsection{An obstruction to the splitting}\label{An obstruction to the splitting}

We return to the framework of the tower constructed in Theorem~\ref{TheMainResult}. As mentioned in the introduction Theorem~\ref{TheMainResult} can be thought of as a generalization of the Miller splitting \cite{Miller}. Thus there is interest in determining whether the tower could possibly split stably, to answer this question we prove Theorem~\ref{IntroObstruction}. 

Recall the triangles
\[
\xymatrix@C=1.5cm
{X_{k}\ar[d]_{\pi_{k}}&G_{k}(V_{0})^{\Hom(T,V_{1} - V_{0})\oplus s(T)}\ar[l]_{\phi_{k}\phantom{xxxxxxx}}\\
X_{k-1}\ar[ur]|\bigcirc_{\phantom{xxx}\delta_{k}}&}
\]
and observe that a splitting is only possible if all maps $\delta_{k}$ are null homotopic. We provide an interesting geometric description of $\delta_{1}:S^{0}\to \Sigma PV_{0}^{\Hom(T,V_{1} - V_{0})\oplus s(T)}$ and use this to produce a cohomological obstruction to a splitting in most cases.

Note that over $PV_{0}$ the bundle $s(T)$ is a copy of the trivial bundle $\Real$. We study the map $\delta_{1}\co S^{0}\to PV_{0}^{\Real\oplus\Hom(T,V_{1}-V_{0})\oplus s(T)}\cong \Sigma^{2} PV_{0}^{\Hom(T,V_{1}-V_{0})}$. Let $j\co PV_{0}\rightarrowtail\svo$ be the equivariant embedding $L\mapsto 0|_{L}\oplus -1|_{L^{\bot}}$. As covered in the previous section we have a stable Pontryagin--Thom collapse map
\[
\Sigma^{-\svo}j^{!}\co S^{0}\to PV_{0}^{-\tau_{PV_{0}}}.
\]
It is well known that the tangent bundle of $PV_{0}$ is the bundle $\Hom(T,T^{\bot})$. Recall the equivariant identity $\Hom(T,T^{\bot})\cong \Hom(T,V_{0})\ominus 2.s(T)$ mentioned in the proof of \ref{equibundleids}. This allows us to rewrite the collapse as
\[
\Sigma^{-\svo}j^{!}\co S^{0}\to PV_{0}^{2.s(T)-\Hom(T,V_{0})}\cong \Sigma^{2}PV_{0}^{-\Hom(T,V_{0})}.
\]
Let $i_{zero}\co PV_{0}^{-\Hom(T,V_{0})}\to PV_{0}^{\Hom(T,V_{1}-V_{0})}$ be the stable zero section and define $\delta^{j}$ to be the composition
\[
\delta^{j}\co S^{0}\overset{\Sigma^{-W}j^{!}}{\to}\Sigma^{2}PV_{0}^{-\Hom(T,V_{0})}\overset{-\Sigma(\Sigma i_{zero})}{\to}\Sigma^{2} PV_{0}^{\Hom(T,V_{1}-V_{0})}.
\]

\begin{proposition}\label{bottommapisPT} $\delta_{1}\simeq \delta^{j}$.
\end{proposition}
\begin{proof} We prove this unstably. Let $\tilde{i}_{zero}\co PV_{0}^{s(T^{\bot})}\to PV_{0}^{\Hom(T,V_{1})\oplus s(T^{\bot})}$ be the zero section. We have a composition
\[
\tilde{\delta^{j}}\co S^{\svo}\overset{j^{!}}{\to} PV_{0}^{\svo-\Hom(T,T^{\bot})}\cong \Sigma PV_{0}^{s(T^{\bot})}\overset{-\Sigma \tilde{i}_{zero}}{\to}\Sigma PV_{0}^{\Hom(T,V_{1})\oplus s(T^{\bot})}.
\]
It is clear that $\Sigma^{-\svo}\tilde{\delta^{j}}=\delta^{j}$. Now let $p_{0}$ be the collapse $p_{0}\co S^{\svo}\to S^{\svo}/\sim$ where $\alpha\sim\alpha'$ if and only if $e_{d_{0}-1}(\alpha)=e_{d_{0}-2}(\alpha)$ and $e_{d_{0}-1}(\alpha')=e_{d_{0}-2}(\alpha')$. Further, define
\begin{align*}
m\co S^{\svo}/\sim&\to \Sigma PV_{0}^{s(T^{\bot})}\\
\alpha&\mapsto (e_{d_{0}-1}(\alpha),\Ker(\alpha-e_{d_{0}-1}(\alpha)),-\log(e_{d_{0}-1}(\alpha)-\alpha)|_{\Ker(\alpha-e_{d_{0}-1}(\alpha))^{\bot}}).
\end{align*}
By following through with the definition of $\tilde{\delta_{1}}$ from \ref{themapdeltak} it is easy to see that $\tilde{\delta_{1}}=-\Sigma\tilde{i}_{zero}\circ m\circ p_{0}$. Hence the claim follows if $j^{!}\simeq m\circ p_{0}$. This, however, follows from our specific choice of embedding $j$; it is a simple definition chase to check that the two maps match up.
\end{proof}

Using Proposition~\ref{GysinMapsAndResidue} the below result then follows.

\begin{corollary}\label{ObstructionToSplitting} The map $\delta_{1}$ is zero in equivariant $K$--theory if and only if $f_{V_{0}}(z)$ divides $f_{V_{1}}(z)$.
\end{corollary}

\begin{theorem}\label{NoSplittingThm} Let $G$ be connected, then the tower of Theorem~\ref{TheMainResult} does not split if $V_{0}$ is not a subrepresentation of $V_{1}$.
\end{theorem} 
\begin{proof} From above we know that $\delta_{1}^{*}$ is going to be zero if and only if the meromorphic function $f_{V_{1}}(z)/f_{V_{0}}(z)$ has no singularities and thus if and only if $f_{V_{1}}(z)=f_{V_{0}}(z).g(z)$ for some polynomial $g$. Let $\mathbb{T}$ be a maximal torus of $G$, then it is a classical fact (see \cite[Section $4.3$ and $4.4$]{AtiyahHirzebruchVectorBundles}) that $R(\mathbb{T})$ is a ring of Laurent polynomials over $\mathbb{Z}$ and $R(G)$ is the subring of $R(\mathbb{T})$ of invariants under the action of the Weyl group. It follows as standard that $R(\mathbb{T})$ and $R(G)$ and hence $R(\mathbb{T})[z]$ and $R(G)[z]$ are unique factorization domains. This is enough to complete the proof, this is easiest to see when $G$ is abelian---in this case $V_{0}$ and $V_{1}$ decompose into sums of lines $V_{0}=\bigoplus L_{i}$ and $V_{1}=\bigoplus L_{i}'$. It is a classical fact that the polynomials $f_{V_{0}}(z)$ and $f_{V_{1}}(z)$ factorize as $\prod (z-[L_{i}])$ and $\prod (z-[L'_{i}])$ and the claim then follows.
\end{proof}

Unfortunately this theorem does not as proven pass to the non-connected case; for $G$ some certain finite groups one can choose representations $V_{0}$ and $V_{1}$ such that $V_{0}$ is not a subrepresentation of $V_{1}$ but $f_{V_{0}}(z)$ divides $f_{V_{1}}(z)$. We now consider one case where a splitting may be possible---when $V_{0}\leqslant V_{1}$.

\section{The subrepresentation case---conjecture}\label{The subrepresentation case---conjecture}

We return to the general case of $G$ a general compact Lie group. As indicated in the previous section the tower does not in general split if $V_{0}$ is not a subrepresentation of $V_{1}$. Consider instead the case where $V_{0}\leqslant V_{1}$, i.e.\ $V_{1}=V_{0}\oplus V_{2}$ for some representation $V_{2}$ and we have an inclusion $I\co V_{0}\to V_{1}$. Miller built a stable splitting of $\Ell_{\infty}$ in \cite{Miller} by first building a filtration
\[
F_{k}(\Ell):=\{\theta\in\Ell:\text{rank}(\theta-I)\leqslant k\}.
\]
The inclusion $F_{k-1}(\Ell)\rightarrowtail F_{k}(\Ell)$ is a cofibration and hence we have a cofibre sequence
\[
F_{k-1}(\Ell)_{\infty}\rightarrowtail F_{k}(\Ell)_{\infty}\twoheadrightarrow\frac{F_{k}(\Ell)_{\infty}}{F_{k-1}(\Ell)_{\infty}}\to\ldots.
\]
Miller completes the proof by building a homeomorphism and splitting map
\[
\tau_{k}\co \frac{F_{k}(\Ell)_{\infty}}{F_{k-1}(\Ell)_{\infty}}\overset{\cong}{\to} G_{k}(V_{0})^{\Hom(T,V_{2})\oplus s(T)},
\]
\[
\sigma_{k}\co G_{k}(V_{0})^{\Hom(T,V_{2})\oplus s(T)}\to F_{k}(\Ell)_{\infty}.
\]
It follows that there are stable splittings
\[
F_{k}(\Ell)_{\infty}\simeq F_{k-1}(\Ell)_{\infty}\vee G_{k}(V_{0})^{\Hom(T,V_{2})\oplus s(T)}.
\]
We conjecture that we can recover a similar stable splitting from our tower; thus the tower can be thought of as the `other direction' of the Miller splitting. We have a composition $F_{k}(\Ell)_{\infty}\rightarrowtail \Ell_{\infty}\twoheadrightarrow X_{k}$, call this map $r_{k}$.

\begin{conjecture}\label{theMillerconjecture}  $F_{k}(\Ell)_{\infty}\simeq X_{k}$ via $r_{k}$.
\end{conjecture}

In \cite[Section $7.3$]{HarryPhD} this conjecture is shown to be equivalent to proving that there is a homotopy 
\[
\sigma_{k}\circ r_{k}\simeq \phi_{k}\co G_{k}(V_{0})^{\Hom(T,V_{1}-V_{0})\oplus s(T)}=G_{k}(V_{0})^{\Hom(T,V_{2})\oplus s(T)}\to X_{k},
\]
however, both conjectures have proved to be surprisingly hard to solve. A new, different, formulation of some of the ideas in preparation by the author and Strickland may lead to a way forward, but under the current formulation we can only provide partial splitting results. Were the conjecture to hold then we can retrieve a splitting as below.

\begin{proposition}\label{retrievingthemillersplitting} Assume \ref{theMillerconjecture} holds, then we have an equivariant splitting
\[
X_{k}\simeq X_{k-1}\vee G_{k}(V_{0})^{\Hom(T,V_{2})\oplus s(T)}.
\]
\end{proposition}
\begin{proof} We have a commutative diagram
\[
\xymatrix{F_{k}(\Ell)_{\infty}\ar[r]^{\phantom{xxxxxx}r_{k}}&X_{k}\ar[d]^{\pi_{k}}\\
F_{k-1}(\Ell)_{\infty}\ar[u]_{i_{k}}\ar[r]_{\phantom{xxxxxx}r_{k-1}}&X_{k-1}}
\]
where $i_{k}$ is the standard inclusion. Let $r_{k-1}^{-1}$ denote our homotopy inverse, then we have a composition $r_{k}\circ i_{k}\circ r_{k-1}^{-1}\co X_{k-1}\to X_{k}$. Consider the self-map $\pi_{k}\circ r_{k}\circ i_{k}\circ r_{k-1}^{-1}\co X_{k-1}\to X_{k-1}$, as the above diagram commutes we have $\pi_{k}\circ r_{k}\circ i_{k}\circ r_{k-1}^{-1}\simeq r_{k-1}\circ r_{k-1}^{-1}\simeq\text{id}_{X_{k-1}}$ and hence $r_{k}\circ i_{k}\circ r_{k-1}^{-1}$ is a splitting map.
\end{proof}

\subsection{The subrepresentation case---results}\label{The subrepresentation case---results}

Although we can't demonstrate the splitting in general we can prove that the top and bottom portions of the tower split in the subrepresentation case, thus proving Theorem~\ref{IntroSplitting}.

\begin{proposition}\label{thebottomofthetowersplits} We have an equivariant splitting
\[
X_{1}\simeq S^{0}\vee PV_{0}^{\Hom(T,V_{2})\oplus s(T)}.
\]
\end{proposition}
\begin{proof} Note that $r_{0}\co S^{0}\to S^{0}$ is the identity, hence it has an inverse. The techniques of \ref{retrievingthemillersplitting} then produce the result.
\end{proof}

We would hope to extend this result to producing a complete Miller splitting via an inductive argument. While using \ref{retrievingthemillersplitting} we can build a splitting from the assumption that $r_{k-1}$ is an equivalence we cannot from this assumption show that $r_{k}$ is an equivalence, a needed fact to complete the induction. Hence we seem to be unable to generalize this result to producing a complete splitting. 

\begin{proposition}\label{thetopofthetowersplits} The map $r_{d_{0}-1}$ provides a stable homotopy equivalence $F_{d_{0}-1}(\Ell)_{\infty}\simeq X_{d_{0}-1}$ and hence there is an equivariant splitting
\[
\Ell_{\infty}\simeq X_{d_{0}-1}\vee G_{d_{0}}(V_{0})^{\Hom(T,V_{2})\oplus s(T)}.
\]
\end{proposition}
\begin{proof} For shorthand write $G_{d_{0}}$ for $G_{d_{0}}(V_{0})^{\Hom(T,V_{2})\oplus s(T)}$. We observe that the description of our map $\phi_{d_{0}}\co
G_{d_{0}} \to \Ell_{\infty}$ matches exactly the description of the Miller splitting map $\sigma_{d_{0}}$ (compare \ref{thetaEalpha} to \cite[Lemma $1.1$]{Crabb} or \cite[Lemma $1.3$]{Kitchloo}). It follows that we have a diagram
\[
\xymatrix{&&\text{pt}\ar[dr]|\bigcirc&&\\
&G_{d_{0}}\ar[ur]\ar[dl]|\bigcirc&&\ar[ll]_{\text{id}}G_{d_{0}}\ar[dl]^{\phi_{d_{0}}}&\\
F_{d_{0}-1}(\Ell)_{\infty}\ar@/_1pc/[rrrr]_{r_{d_{0}-1}}\ar[rr]&&\Ell_{\infty}\ar[ul]\ar[rr]&&X_{d_{0}-1}\ar[ul]|\bigcirc}
\]
of cofibre sequences. By the octahedral axiom we retrieve that $r_{d_{0}-1}$ has contractible cofibre. As we are working stably, $r_{d_{0}-1}$ gives an isomorphism of homotopy groups and as our spectra are $G$--$CW$--spectra it follows that $r_{d_{0}-1}$ is an equivalence by the Whitehead Theorem.
\end{proof}
Again, the techniques in this proof fail to generalize to a proof that the tower splits. We can, however, combine the results to produce an equivariant splitting in a special case.

\begin{theorem}\label{Dimntwosplitting} Let $V_{1}=V_{0}\oplus V_{2}$ and let $d_{0}=2$. Then the tower of Theorem~\ref{TheMainResult} splits to produce an equivariant Miller splitting
\[
\Ell_{\infty}\simeq \bigvee_{k=0}^{2}G_{d_{0}}(V_{0})^{\Hom(T,V_{2})\oplus s(T)}.
\]
\end{theorem}

\bibliography{EquiGeneralBib}

\bibliographystyle{alpha}

\end{document}